\documentclass[11pt]{amsart}
\usepackage[leqno]{amsmath}
\usepackage{amsmath,amsfonts,amsthm,amssymb}
\usepackage{epsfig}
\usepackage{tikz}
\usepackage{graphicx}
\usepackage[labelfont=rm]{subcaption}
\usepackage{cite,url}
\usepackage{hyperref}
\usepackage[left=1in,right=1in,top=1in,bottom=1in,foot=0.5in]{geometry}
\usepackage[T1]{fontenc}
\usepackage[utf8]{inputenc}
\usepackage{tikz}
\usepackage[linesnumbered, ruled, vlined]{algorithm2e}
\usepackage{float}
\usepackage{fancyhdr}
\usepackage{enumerate}

\theoremstyle{plain}
\newtheorem{theorem}{Theorem}
\newtheorem{lemma}{Lemma}
\newtheorem{proposition}{Proposition}

\newtheorem{claim}{Claim}

\newtheorem{case}{Case}

\theoremstyle{definition}

\newtheorem{question}[theorem]{Question}

\newtheorem{example}[theorem]{Example}

\numberwithin{equation}{section}

\definecolor{airforceblue}{rgb}{0.36, 0.54, 0.66}

\DeclareMathOperator{\B}{B}
\DeclareMathOperator{\E}{E}

\DeclareMathOperator{\conv}{conv}

\DeclareMathOperator*{\qr}{qr}
\DeclareMathOperator*{\dist}{\sf dist}
\DeclareMathOperator*{\hyperbolic}{\sf Hyperbolic}
\DeclareMathOperator*{\intDeltaSlim}{\sf Interval-Slim}
\DeclareMathOperator*{\isometric}{\sf Isometric}
\DeclareMathOperator*{\subgraph}{\sf Subgraph}
\DeclareMathOperator*{\acyclicity}{\sf Acyclicity}
\DeclareMathOperator*{\bipartite}{\sf Bipartite}
\DeclareMathOperator*{\even}{\sf Even}
\DeclareMathOperator*{\connectivity}{\sf Connectivity}
\DeclareMathOperator*{\tree}{\sf Tree}
\DeclareMathOperator*{\tri}{\sf triangle}
\DeclareMathOperator*{\squ}{\sf square}
\DeclareMathOperator*{\pentagon}{\sf pentagon}
\DeclareMathOperator*{\tc}{\sf Triangle-Condition}
\DeclareMathOperator*{\qc}{\sf Quadrangle-Condition}
\DeclareMathOperator*{\aqc}{\sf Almost-Quadrangle-Condition}
\DeclareMathOperator*{\pc}{\sf Pentagon-Condition}
\DeclareMathOperator*{\tpc}{\sf Triangle-Pentagon-Condition}
\DeclareMathOperator*{\inc}{\sf Interval-Neighborhood-Condition}
\DeclareMathOperator*{\ii}{\sf closest}
\DeclareMathOperator*{\boundary}{\sf InBdy}
\DeclareMathOperator*{\intBoundary}{\sf InIntBdy}
\DeclareMathOperator*{\phstable}{\sf ph-stable}
\DeclareMathOperator*{\degTconv}{\sf deg-3-conv}
\DeclareMathOperator*{\mtriangle}{\sf metric-triangle}

\DeclareMathOperator*{\colinear}{\sf colinear}
\DeclareMathOperator*{\quasimed}{\sf quasi-median}
\DeclareMathOperator*{\med}{\sf median}

\DeclareMathOperator*{\convW}{\sf convex-halfspace}
\DeclareMathOperator*{\convnonW}{\sf convex-complement}
\DeclareMathOperator*{\convWeq}{\sf convex-halfspace-eq}
\DeclareMathOperator*{\convnonWeq}{\sf convex-complement-eq}

\DeclareMathOperator*{\gated}{\sf gated-int}
\DeclareMathOperator*{\cube}{\sf cube-int}
\DeclareMathOperator*{\antipodaluv}{\sf antipodal-int}

\DeclareMathOperator*{\stronglyequilateral}{\sf strongly-equilateral}
\DeclareMathOperator*{\intIncl}{\sf int-incl}

\DeclareMathOperator{\poly}{poly}

\newcommand{\thick}{\textsf{Thick}\xspace}

\newcommand{\poc}{\textsf{Positioning-Condition}\xspace}
\newcommand{\tict}{\textsf{2-Interval-Condition-3}\xspace}
\newcommand{\ticq}{\textsf{2-Interval-Condition-4}\xspace}

\newcommand{\wmodular}{\textsf{Weakly Modular}\xspace}
\newcommand{\modular}{\textsf{Modular}\xspace}
\newcommand{\qmodular}{\textsf{Quasi-Modular}\xspace}
\newcommand{\pmodular}{\textsf{Pseudo-Modular}\xspace}
\newcommand{\meshed}{\textsf{Meshed}\xspace}

\newcommand{\strongEquil}{\textsf{Strongly-Equilateral Triangles}\xspace}

\newcommand{\median}{\textsf{Median}\xspace}
\newcommand{\pmedian}{\textsf{Pseudo-Median}\xspace}
\newcommand{\qmedian}{\textsf{Quasi-Median}\xspace}
\newcommand{\wmedian}{\textsf{Weakly Median}\xspace}
\newcommand{\amedian}{\textsf{Almost Median}\xspace}
\newcommand{\netpcube}{\textsf{Netlike Partial Cube}\xspace}

\newcommand{\bridged}{\textsf{Bridged}\xspace}
\newcommand{\wbridged}{\textsf{Weakly Bridged}\xspace}
\newcommand{\cballs}{\textsf{Convex Balls}\xspace}
\newcommand{\bucolic}{\textsf{Bucolic}\xspace}

\newcommand{\helly}{\textsf{Helly}\xspace}
\newcommand{\chelly}{\textsf{Clique-Helly}\xspace}
\newcommand{\cqwq}{\textsf{C4W4}\xspace}

\newcommand{\dualpolar}{\textsf{Dual Polar}\xspace}
\newcommand{\swmodular}{\textsf{Sweakly Modular}\xspace}
\newcommand{\stmodular}{\textsf{Strongly Modular}\xspace}

\newcommand{\matroid}{\textsf{Matroids}\xspace}
\newcommand{\deltamatroid}{\textsf{$\Delta$-Matroids}\xspace}

\newcommand{\pcube}{\textsf{Partial Cube}\xspace}
\newcommand{\phamming}{\textsf{Partial Hamming}\xspace}

\newcommand{\ample}{\textsf{Ample}\xspace}
\newcommand{\om}{\textsf{Oriented Matroid}\xspace}
\newcommand{\com}{\textsf{COM}\xspace}
\newcommand{\pasch}{\textsf{Pasch}\xspace}
\newcommand{\peano}{\textsf{Peano}\xspace}
\newcommand{\convint}{\textsf{Convex Intervals}\xspace}
\newcommand{\sandglass}{\textsf{Sand Glass}\xspace}
\newcommand{\sepvv}{\textsf{S2separability}\xspace}
\newcommand{\sepvc}{\textsf{S3separability}\xspace}
\newcommand{\sepcc}{\textsf{S4separability}\xspace}
\newcommand{\jhc}{\textsf{JHC}\xspace}
\newcommand{\cellular}{\textsf{Cellular}\xspace}

\newcommand{\antipodalglobal}{\textsf{Antipodal}\xspace}
\newcommand{\wmpasch}{\textsf{Weakly Modular Pasch}\xspace}
\newcommand{\bpasch}{\textsf{Bipartite Pasch}\xspace}
\newcommand{\paschpeano}{\textsf{Pasch-Peano}\xspace}
\newcommand{\bpeano}{\textsf{Bipartite Peano}\xspace}

\newcommand{\block}{\textsf{Block Graph}\xspace}
\newcommand{\dhg}{\textsf{Distance Hereditary}\xspace}
\newcommand{\pt}{\textsf{Ptolemaic}\xspace}

\newcommand{\metric}{\textsf{Metric}\xspace}

\newcommand{\chordal}{\textsf{Chordal}\xspace}
\newcommand{\planar}{\textsf{Planar}\xspace}
\newcommand{\partialJ}{\textsf{Partial Johnson}\xspace}
\newcommand{\dismantlable}{\textsf{Dismantlable}\xspace}
\newcommand{\eulerian}{\textsf{Eulerian}\xspace}
\newcommand{\dpo}{\textsf{DPO}\xspace}

\newcommand{\PoC}{\mathrm{PosC}\xspace}
\newcommand{\TiCT}{\mathrm{2IC3}\xspace}
\newcommand{\TiCQ}{\mathrm{2IC4}\xspace}
\newcommand{\LiC}{\mathrm{LC}\xspace}

 \newcommand{\INC}{\mathrm{INC}}
\newcommand{\TC}{\mathrm{TC}}
\newcommand{\QC}{\mathrm{QC}}
\newcommand{\AQC}{\mathrm{AQC}}

 \newcommand{\TPC}{\mathrm{TPC}}

\newcommand{\bA}{\mathbf{A}}
\newcommand{\bB}{\mathbf{B}}
\newcommand{\bG}{\mathbf{G}}
\newcommand{\C}{\mathcal{C}}

\begin{document}

\medskip

\centerline{\Large\bf First-Order Logic Axiomatization of Metric Graph Theory}

\vspace{6mm}
\centerline{J\'er\'emie Chalopin, Manoj Changat, Victor Chepoi, and  Jeny Jacob}

\vspace{4mm}
\begin{small}
\centerline{Aix-Marseille Universit\'e, CNRS, Universit\'e de Toulon, LIS, Marseille, France}
\centerline{\texttt{\{jeremie.chalopin,victor.chepoi\}@lis-lab.fr}}
\end{small}

\vspace{4mm}
\begin{small}
\centerline{University of Kerala, Department of Futures Studies,  Trivandum, India}
\centerline{\texttt{\{mchangat,jenyjacobktr\}@gmail.com}}
\end{small}

\medskip

\bigskip\noindent
{\footnotesize {\bf Abstract.} The main goal of this note is to provide a \emph{First-Order Logic with Betweenness (FOLB)}
axiomatization of the main classes of graphs occurring in Metric Graph Theory, in analogy to Tarski's axiomatization
of Euclidean geometry. We provide such an axiomatization for weakly modular graphs and their principal subclasses (median and modular graphs,
bridged graphs, Helly graphs, dual polar graphs, etc), basis graphs of matroids and even $\Delta$-matroids, partial cubes
and their subclasses (ample partial cubes, tope graphs of oriented matroids and complexes of oriented matroids, bipartite Pasch and Peano graphs, cellular and hypercellular partial cubes,
almost-median graphs, netlike partial cubes), and Gromov hyperbolic graphs. On the other hand, we show that some classes of graphs (including chordal,
planar, Eulerian, and dismantlable graphs), closely related with  Metric Graph Theory,
but defined in a combinatorial or topological way, do not allow such an axiomatization.}

\section{Introduction}

\emph{First-Order Logic (FOL)} is the language of classical logic most widely used  in various areas of mathematics and computer science.
First-order logic uses quantified variables over non-logical objects and allows the use of sentences that contain variables.
The {\it first-order language} of graph theory  is defined in the usual way with variables ranging over the vertex-set
and the edge relation as the primitive relation. However, not many graph properties can
be expressed using this logic: such fundamental properties  as  Connectivity, Acyclicity, Bipartiteness, Planarity,
Eulerian, and Hamiltonian Path  are not first-order definable on finite graphs~\cite{Kol,hamiltonian}.
Therefore developing a comprehensive first-order theory on graphs with more expressive power is an important problem.
A possible approach towards such a theory  is to transpose  to graphs
Tarski's axiomatic approach to Euclidean geometry~\cite{Ta,TaGi,SchSzmTa}.

Tarski developed a First-Order Logic theory of Euclidean geometry using only ``points'' as the  ``primitive geometric objects'' in contrast to other theories of Euclidean geometry of Hilbert and Birkhoff,
where points, lines, planes, etc., are all primitive ``geometrical objects''.  In Tarski's theory,
there are two primitive geometrical relations (predicats): the ternary relation $B$ of ``betweenness'' and the quaternary relation $\equiv$ of ``equidistance''
or ``congruence of segments''. The elegance of Tarski's axiomatic theory of geometry is that the axiom system admits elimination of quantifiers:
that is, every formula is provably equivalent (on the basis of the axioms) to a Boolean combination of basic formulas. The theory is complete:
every assertion is either provable or refutable; the theory is decidable -- there is a mechanical procedure for determining whether or not any given
assertion is provable and also there is a constructive consistency proof for the theory.
Tarski's axioms  are an axiom set for the substantial fragment of Euclidean geometry that is formulable in first-order logic with identity, and requiring
no set theory~\cite{Ta,TaGi,SchSzmTa}.

The main goal of this article is the {\it First-Order Logic axiomatization of Metric Graph Theory} using the notion of \emph{Betweenness}, in a similar vein as Tarski's First-Order Logic approach to  Euclidean geometry.  The natural betweenness on graphs is the {\it metric betweenness}  (or {\it shortest path betweenness})
resulting from the standard path metric $d$ of
connected graphs  $G=(V,E)$ and  defined using the ternary relation
$\B(abc)$ on the vertex set $V$ of a graph $G$ meaning that ``the vertex $b$ lies on some shortest path of $G$ between the vertices $a$ and $c$''. We abbreviate the \emph{Fist Order Logic with Betweenness} of graphs by \emph{FOLB}. 

The main subject of {\it Metric Graph
Theory} (MGT) is the investigation and characterization of graph classes
whose metric satisfies main properties of
classical metric geometries  like Euclidean $\ell_2$-geometry (and more generally, the $\ell_1$- and $\ell_{\infty}$-geometries),
hyperbolic spaces, hypercubes, trees. Such central properties are
convexity of balls, Helly property for balls, geometry of geodesic or
metric triangles, isometric and low-distortion embeddings, the retractions,
the four-point conditions, uniqueness or existence of medians, etc.
The main classes of graphs central to MGT are median graphs, Helly graphs, partial cubes and
${\ell}_1$--graphs, bridged graphs, graphs with convex balls, Gromov hyperbolic graphs, modular and weakly modular graphs.
Other classes surprisingly arise  from  combinatorics and geometry: basis graphs
of matroids, even $\Delta$-matroids, tope graphs of
oriented matroids, dual polar spaces.  For a survey of classes arising in MGT,  see the survey~\cite{BaCh_survey}). For a theory
of weakly modular graphs and their subclasses, see the paper~\cite{CCHO} and for partial cubes and $\ell_1$-graphs, see the
books~\cite{DeLa} and~\cite{HaImKl}.

In this paper, we show that all these graph classes occurring in MGT
are definable in FOLB. On the other hand, we show that chordal graphs,
dismantlable graphs, Eulerian graphs, planar graphs, and partial
Johnson graphs are not definable in FOLB. Chordal graphs form a
subclass of bridged graphs, dismantlable graphs form a superclass of
bridged and Helly graphs, partial Johnson graphs generalize partial
cubes. Since often the FOLB-definability of a graph class is based not
on its initial definition but on a characterization, which is not the
principal or nicest one, when we introduce a graph class we define it,
briefly motivate its importance, and present the used
characterization. Then we refer to mentioned above papers and books or
to other references for a more detailed treatment of each
class. Notice that many of the classes from metric graph theory
contain all trees but not all cycles; they are often defined by
forbidding isometric subgraphs and cycles of given lengths.
Consequently, these classes cannot be defined in the standard First
Order Logic on graphs: indeed, the proof establishing that
$\acyclicity$ is not FOL-definable immediately implies that such
classes are not FOL-definable.

The paper is organized as follows. In Section \ref{sec:preliminaries} we present the main basic definitions about graphs and First Order Logic.
We also recall the Ehrenfeucht-Fra\"{i}ss\'e games as the tool of proving that some queries are not definable in FOL for graphs.
In Section \ref{sec:FOLB} we introduce the First Order Logic
with Betweenness for graphs and give the first examples of queries which are definable in this logic. In Section \ref{sec:weakly-modular} we show that weakly modular
graphs and their main subclasses and super-classes occurring in Metric Graph Theory are FOLB-definable. In Section \ref{sec:partial-cubes} we show FOLB-definability
of partial cubes and some of their subclasses and super-classes. Section \ref{sec:hyperbolicity} is devoted to FOLB-definability of Gromov hyperbolic graphs. In Section~\ref{notfolb}, we establish that some classes of graphs are not FOLB-definable. In Section~\ref{FO-poly}, we explain how such FOLB characterization lead to polynomial time recognition algorithms.

\section{Preliminaries}\label{sec:preliminaries}

In this section, we recall the main definitions about graphs and the first-order logic. In the three subsections about first-order logic we
closely follow the paper~\cite{Kol} by Kolaitis  (we also use the book by Libkin~\cite{Lib}).

\subsection{Graphs}
A \emph{graph} $G=(V,E)$ consists of a set of vertices $V:=V(G)$ and a set of edges $E:=E(G)\subseteq V\times V$. All graphs considered in this paper are
finite, undirected, connected, and contain no multiple edges nor loops. For two distinct vertices $v,w\in V$ we write $v\sim w$
(respectively, $v\nsim w$) when there is an (respectively, there is no) edge connecting
$v$ with $w$, that is, when $vw:=\{v,w \} \in E$.
For vertices $v,w_1,\ldots,w_k$, we write $v\sim w_1,\ldots,w_k$ (respectively, $v\nsim w_1,\ldots,w_k$) or $v\sim A$ (respectively,
$v\nsim A$) when $v\sim w_i$ (respectively, $v\nsim w_i$), for each $i=1,\ldots, k$, where $A=\{ w_1,\ldots,w_k\}$.
As maps between graphs $G=(V,E)$ and $G'=(V',E')$ we always consider \emph{simplicial maps}, that is functions of the form
$f\colon V\to V'$ such that if $v\sim w$ in $G$ then $f(v)=f(w)$ or $f(v)\sim f(w)$ in $G'$.
A $(u,w)$--path $(v_0=u,v_1,\ldots,v_k=w)$ of \emph{length} $k$ is a sequence of vertices with $v_i\sim v_{i+1}$. If $k=2,$
then we call $P$ a \emph{2-path} of $G$. If $v_i\ne v_j$ for $|i-j|\ge 1$, then $P$ is
called a \emph{simple $(a,b)$--path}.
A $k$--cycle $(v_0,v_1,\ldots,v_{k-1})$ is a path $(v_0,v_1,\ldots,v_{k-1},v_0)$.
For a subset
$A\subseteq V,$ the subgraph of $G=(V,E)$  \emph{induced by} $A$
is the graph $G(A)=(A,E')$ such that $uv\in E'$ if and only if $uv\in E$
($G(A)$ is sometimes called a \emph{full subgraph} of $G$).
A \emph{square} $uvwz$ (respectively, \emph{triangle} $uvw$, \emph{pentagon} $uvwxz$) is an induced $4$--cycle $(u,v,w,z)$ (respectively, $3$--cycle $(u,v,w)$, $5$-cycle $(u,v,w,x,z)$).

The \emph{distance} $d(u,v)=d_G(u,v)$ between two vertices $u$ and $v$ of a graph $G$ is the
length of a shortest $(u,v)$--path.  For a vertex $v$ of $G$ and an integer $r\ge 1$, we  denote  by $B_r(v,G)$ (or by $B_r(v)$)
the \emph{ball} in $G$
(and the subgraph induced by this ball)  of radius $r$ centered at  $v$, that is,
$B_r(v,G)=\{ x\in V: d(v,x)\le r\}.$ More generally, the $r$--\emph{ball  around a set} $A\subseteq V$
is the set (or the subgraph induced by) $B_r(A,G)=\{ v\in V: d(v,A)\le r\},$ where $d(v,A)=\mbox{min} \{ d(v,x): x\in A\}$.
As usual, $N(v)=B_1(v,G)\setminus\{ v\}$ denotes the set of neighbors of a vertex
$v$ in $G$. A graph $G=(V,E)$
is \emph{isometrically embeddable} into a graph $H=(W,F)$
if there exists a mapping $\varphi : V\rightarrow W$ such that $d_H(\varphi (u),\varphi
(v))= d_G(u,v)$ for all vertices $u,v\in V$. More generally, for an integer $k\ge 1$, a graph $G=(V,E)$
is \emph{scale $k$ embeddable} into a graph $H=(W,F)$
if there exists a mapping $\varphi : V\rightarrow W$ such that $d_H(\varphi (u),\varphi
(v))=k \cdot d_G(u,v)$ for all vertices $u,v\in V$. A \emph{retraction}
$\varphi$ of a graph $G$ is an idempotent nonexpansive mapping of $G$ into
itself, that is, $\varphi^2=\varphi:V(G)\rightarrow V(G)$ with $d(\varphi
(x),\varphi (y))\le d(x,y)$ for all $x,y\in W$.  The subgraph of $G$
induced by the image of $G$ under $\varphi$ is referred to as a \emph{retract} of $G$.

The \emph{interval}
$I(u,v)$ between $u$ and $v$ consists of all vertices on shortest
$(u,v)$--paths, that is, of all vertices (metrically) \emph{between} $u$
and $v$: $I(u,v)=\{ x\in V: d(u,x)+d(x,v)=d(u,v)\}.$  If $d(u,v)=2$, then $I(u,v)$ is called a \emph{2-interval}.
A 2-interval $I(u,v)$ is called \emph{thick} if $I(u,v)$ contains two non-adjacent vertices $x,y\in I(u,v)\setminus \{ u,v\}$.
A graph $G$ is called \emph{thick} if all 2-intervals of $G$ are thick.   A subgraph of $G$ (or the corresponding vertex-set $A$) is called \emph{convex}
if it includes the interval of $G$ between any pair of its
vertices. The smallest convex subgraph containing a given subgraph $S$
is called the \emph{convex hull} of $S$ and is denoted by conv$(S)$.
A \emph{halfspace} is a convex set of $G$ whose complement is convex. An induced subgraph $H$ (or the corresponding vertex-set of $H$)
of a graph $G$
is \emph{gated} if for every vertex $x$ outside $H$ there
exists a vertex $x'$ in $H$ (the \emph{gate} of $x$)
such that  $x'\in I(x,y)$ for any $y$ of $H$. Gated sets are convex and
the intersection of two gated sets is
gated. By Zorn's lemma there exists a smallest gated subgraph
containing a given subgraph $S$, called the
\emph{gated hull} of $S$.

Let $G_{i}$, $i \in \Lambda$ be an arbitrary family of graphs. The
\emph{Cartesian product} $\prod_{i \in \Lambda} G_{i}$ is a graph whose vertices
are all functions $x: i \mapsto x_{i}$, $x_{i} \in V(G_{i})$ and
two vertices $x,y$ are adjacent if there exists an index $j \in \Lambda$
such that $x_{j} y_{j} \in E(G_{j})$ and $x_{i} = y_{i}$ for all $i
\neq j$. Note that a Cartesian product of infinitely many nontrivial
graphs is disconnected. Therefore, in this case the connected
components of the Cartesian product are called \emph{weak Cartesian products}.
The \emph{direct
product}  $\boxtimes_{i \in \Lambda} G_{i}$  of graphs  $G_{i}$, $i \in \Lambda$
is a graph having the same set of vertices as the Cartesian product
and two vertices $x,y$ are adjacent if
$x_{i} y_{i} \in E(G_{i})$ or $x_{i} = y_{i}$ for all $i\in \Lambda$.

We continue with the definition of some graphs. The complete graph on $n$ vertices is denote by $K_n$ and the complete bipartite graph with parts of size $n$ and $m$ by $K_{n,m}$.
The {\em wheel} $W_k$ is a graph obtained by connecting a single
vertex -- the {\em central vertex} $c$ -- to all vertices of the
$k$--cycle $(x_1,x_2, \ldots, x_k)$; the {\it almost wheel}
$W_k^{-}$ is the graph obtained from $W_k$ by deleting a spoke (i.e.,
an edge between the central vertex $c$ and a vertex $x_i$ of the
$k$--cycle). Analogously $K^-_4$ and $K^-_{3,3}$ are the graphs obtained from $K_4$ and $K_{3,3}$ by removing one edge.
An $n$--{\it octahedron} $K_{n\times 2}$ (or, a {\it hyperoctahedron}, for short) is the complete graph $K_{2n}$
on $2n$ vertices minus a perfect matching.
A {\it hypercube} $Q_m$ of dimension $m$ is a graph having the  subsets of a set $X$ of size $m$ as vertices
and two  such sets $A,B$ are adjacent in $Q_m$
if and only if $|A\triangle B|=1$. A {\it halved cube} $\frac{1}{2}Q_m$ has the vertices of a hypercube $Q_m$ corresponding to subsets of $X$ of even cardinality
as vertices and two  such vertices are adjacent in $\frac{1}{2}Q_m$  if and only if their  distance
in $Q_m$ is 2 (analogously one can define a halved cube on finite subsets of odd cardinality).
For a positive integer $k$, the {\it Johnson graph}  $J(m,k)$ has the subsets of $X$ of size $k$ as vertices and two such vertices are adjacent in $J(m,k)$
if and only if their  distance in $Q_m$ is $2$.  All Johnson graphs $J(m,k)$ with even $k$ are isometric subgraphs of the halved cube $\frac{1}{2}Q_m$ and the
halved cube $\frac{1}{2}Q_m$ is scale 2 embedded in the hypercube $Q_m$.
The hypercube $Q_m$ can be viewed as the Cartesian product of $m$ copies of $K_2$. The \emph{Hamming graph} $H_{m_1,\ldots,m_d}$ is a Cartesian product of the
complete graphs $K_{m_1},\ldots K_{m_d}$.

\subsection{First-Order Logic (FOL)}\label{fol}
In this subsection we recall the main definitions from First-Order Logic.
A \emph{vocabulary} $\sigma=(P_1,\ldots,c_1,\ldots,c_s)$ consists of a set of \emph{constant symbols} and a set of \emph{relation symbols} (called also \emph{predicates}) of specified arities. Given a vocabulary $\sigma$, the variables and the constant symbols are the  \emph{$\sigma$-terms}.  The set of \emph{formulas} is defined inductively as follows:
\begin{itemize}
\item given terms $t_1,\ldots,t_k$ and a $k$-ary predicate $P$, then $P(t_1,\ldots,t_k)$ is a formula;;
\item for each formulas $F,F'$, $\neg F, (F\wedge F')$, and $F\vee F')$ are formulas;
\item if $F$ is a formula and $x$ is a variable, then $\exists x F$ and $\forall x F$ are formulas.
\end{itemize}
\emph{Atomic formulas} are those constructed according to the first rule. A general \emph{first-order formula} is build
up from atomic formulas using Boolean connectives and the two quantifiers. Given a vocabulary $\sigma$, a \emph{$\sigma$-structure} is a tuple $\bA=(A,P^{\bA}_1,\ldots,P^{\bA}_m, c^{\bA}_1,\ldots c^{\bA}_1)$ consisting of
\begin{itemize}
\item a non-empty set $A$, called the \emph{universe};
\item for each constant $c_i$, an element $c^{\bA}_i$ of $A$;
\item for each $k$-ary predicate $P_i$ in $\sigma$, a $k$-ary relation $P^{\bA}_i\subset \underbrace{A\times\cdots\times A}_k$.
\end{itemize}
A \emph{finite $\sigma$-structure} is a $\sigma$-structure $\bA$ whose universe $A$ is finite.

Let $\bA=(A,P^{\bA}_1,\ldots,P^{\bA}_m, c^{\bA}_1,\ldots,c^{\bA}_s)$ and $\bB=(B,P^{\bB}_1,\ldots,P^{\bB}_m,c^{\bB}_1,\ldots,c^{\bB}_s)$ be two $\sigma$-structures. An \emph{isomorphism} between $\bA$ and $\bB$ is a mapping $h: A\rightarrow B$ that satisfies the following conditions:
\begin{itemize}
\item $h$ is a one-to-one and onto function;
\item for every constant symbol $c_j, j=1,\ldots,s$, we have $h(c^{\bA}_j)=c^{\bB}_j$;
\item for each relation symbol $P_i, i=1,\ldots,m$, of arity $t$ and any $t$-tuple $(a_1,\ldots,a_t)$ from $A$, we have $P^{\bA}_i(a_1,\ldots,a_t)$ if and only if $P^{\bB}_i(h(a_1),\ldots,h(a_t))$.
\end{itemize}

Given two structures  $\mathbf{A}=(A,P^{\bA}_1,\ldots,P^{\bA}_m, c^{\bA}_1,\ldots,c^{\bA}_s)$ and $\bB=(B,P^{\bB}_1,\ldots,P^{\bB}_m,c^{\bB}_1,\ldots,c^{\bB}_s)$, $\bB$ is a \emph{substructure} of $\bA$ if $B\subseteq A$, each $P^{\bB}_i$  is the restriction of $P^{\bA}_i$ to
$B$ (which means that $P^{\bB}_i=P^{\bA}_i\cap B^t$) and $c^{\bB}_j=c^{\bA}_j, j=1,\ldots,s$.  If $\bA$ is a $\sigma$-structure and $D$ is a subset of $A$,
then the \emph{substructure} of $\bA$ \emph{generated} by $D$ is the structure $\bA \upharpoonright D$ having the set $D\cup \{ c^{\bA}_1,\ldots,c^{\bA}_s\}$ as its universe and having the restrictions of the relations
$P^{\bA}_i$ on $D\cup \{ c^{\bA}_1,\ldots,c^{\bA}_s\}$ as its relations. A \emph{partial isomorphism from} $\bA$ to $\bB$ is an isomorphism from a substructure of $\bA$ to a substructure of $\bB$.
Given a structure $\bA$, a variable $x$, and $a\in A$, the structure $\bA_{[x\mapsto a]}$ is the same as $\bA$ except that $x^{\bA_{[x\rightarrow a]}}=a$.

\begin{example}\label{graph} A (undirected) \emph{graph} is a $\sigma$-structure $\bG=(V,E)$ with the vertex-set $V$ as the universe and the vocabulary $\sigma$ with one
binary relation symbol $E$, where $E$ is interpreted as the edge relation.
The subgraph of $\bG$ induced by a set of vertices $D$ of $\bG$ is the substructure of $\bG$ generated by $D$.
\end{example}

Let $\bA$ be a $\sigma$-structure with universe $A$. The \emph{value} $\bA[t]$ of each term $t$ is an element of the universe $A$, inductively defined as follows:
\begin{itemize}
\item for a constant symbol $c$, set $\bA[c]=c^{\bA}$;
\item for a variable $x$, set $\bA[x]=x^{\bA}$;
\item for a term $f(t_1,\ldots,f_k)$, where $f$ is a $k$-ary function symbol and $t_1,\ldots,t_k$ are terms, set $\bA[f(t_1,\ldots,t_k)]=f^{\bA}(\bA[t_1],\ldots,\bA[t_k])$.
\end{itemize}

The \emph{satisfaction relation} $\bA\vDash F$ (which means that $\bA$ satisfies $F$ or that $\bA$ models F) between a $\sigma$-structure $\bA$ and $\sigma$-formula $F$ is defined
by induction over the structure of $F$:
\begin{itemize}
\item $\bA\vDash P(t_1,\ldots,t_k)$ if and only if $(\bA[t_1],\ldots,\bA[t_k])\in P^{\bA}$;
\item $\bA\vDash (F\wedge F')$ if and only if $\bA\vDash F$ and $\bA\vDash F'$;
\item $\bA\vDash (F\vee F')$ if and only if $\bA\vDash F$ or $\bA\vDash F'$;
\item $\bA\vDash \neg F$ if and only if $\bA\nvDash F$;
\item $\bA\vDash \exists x F$ if and only if there exists $a\in A$ such that $\bA_{[x\mapsto a]}\vDash F$;
\item $\bA\vDash \forall x F$ if and only if $\bA_{[x\mapsto a]}\vDash F$ for all $a\in A$;
\item $\bA \vDash t_1=t_2$ if and only if $\bA[t_1]=\bA[t_2]$.
\end{itemize}

A first-order formula $F$ over signature $\sigma$ is \emph{satisfiable} if $\bA\vDash F$ for some $\sigma$-structure $\bA$. If $F$ is not satisfiable it is
called \emph{unsatisfiable}. $F$ is called \emph{valid} if $\bA\vDash F$ for every $\sigma$-structure $\bA$.

Following the terminology of~\cite{Kol,Lib}, we continue with the concept of \emph{query}, one of the most fundamental concepts in finite model theory. Let $\sigma$
be a vocabulary. A \emph{class of $\sigma$-structures} is a collection $\C$ of $\sigma$-structures that is closed under isomorphisms. A \emph{$k$-ary query on}
$\C$ is a mapping $Q$ with domain $\C$ such that $Q$ is preserved under isomorphisms and $Q(\bA)$ is a $k$-ary relation on $\bA$ for all $\bA\in \C$.
A \emph{Boolean query on} $\C$ is a mapping $Q: \C\rightarrow \{ 0,1\}$ that is preserved under isomorphisms. Consequently, $Q$ can be identified with the subclass $\C'=\{ \bA\in \C: Q(\bA)=1\}$ of $\C$.
For example, the $\connectivity$ query on graphs $\bG=(V,E)$ is the Boolean query such that $\connectivity(\bG)=1$ if and only if the graph $\bG$ is connected. Queries are mathematical objects that formalize
the concept of a ``property'' of structures and makes it possible to define what means for such a ``property'' to be expressible in some logic.

Let $L$ be a (first-order) logic and $\C$ a class of $\sigma$-structures. A $k$-ary query $Q$ on $\C$ is \emph{$L$-definable} if there exists a formula $F(x_1,\ldots,x_k)$ of $L$ with $x_1,\ldots,x_k$
as free variables and such that for every $\bA\in \C$, $Q(\bA)=\{ (a_1,\ldots,a_k)\in A^k: \bA\vDash F(a_1,\ldots,a_k)\}$. A Boolean query $Q$ on $\C$ is \emph{$L$-definable} if there exists an $L$-formula $F$
such that for every $\bA\in \C$, $Q(\bA)=1$ if and only if $\bA\vDash F$. Let $L(\C)$ denotes the collection of all $L$-definable queries on $\C$.

The \emph{expressive power} of a logic $L$ on a class $\C$ of finite structures is defined by the collection of $L$-definable queries on $\C$, i.e., is to determine which queries on $\C$ are $L$-definable and which are not.
To show that a query $Q$ is definable, it suffices to find some $L$-formula that defines it on every structure in $\C$. In contrast, showing that $Q$ is not $L$-definable entails showing that no
formula of $L$ defines the property. One of the main tools in proving that a query is not definable in first-order logic of finite graphs is the method of Ehrenfeucht-Fra\"{i}ss\'e games, defined in the next subsection.

\subsection{Ehrenfeucht-Fra\"{i}ss\'e games}\label{EFg}  Let $r$ be a positive integer, $\sigma$ a vocabulary, and $\bA$ and $\bB$ two $\sigma$-structures. The \emph{$r$-move Ehrenfeucht-Fra\"{i}ss\'e game on} $\bA$ and $\bB$
is played between two players, called the \emph{Spoiler} and the \emph{Duplicator}. Each run of the game has $r$ moves. In each move, the Spoiler plays first and picks an element from the universe $A$ of $\bA$ or from the
universe $B$ of $\bB$; the Duplicator then responds by picking an element of the other structure (i.e., if Spoiler picked an element from $A$, then the Duplicator picks and element from $B$, and vice versa). Let
$a_i\in A$ and $b_i\in B$ be the two elements picked by the Spoiler and the Duplicator in their $i$-th move, $1\le i\le r$.
\begin{itemize}
\item The \emph{Duplicator wins the run} $(a_1,b_1),\ldots,(a_r,b_r)$ if the mapping $a_i\mapsto b_i, i=1,\ldots,r$ and $c_j^{\bA}\mapsto c_j^{\bB}, j=1,\ldots,s$ is a partial isomorphism from $\bA$ to $\bB$, which means that it is an isomorphism between the substructure $\bA\upharpoonright \{ a_1,\ldots,a_r\}$ of $\bA$ restricted to $\{ a_1,\ldots,a_r\}$ and the substructure $\bB\upharpoonright \{ b_1,\ldots,b_r\}$ of $\bB$ restricted to $\{ b_1,\ldots,b_r\}$. otherwise, the \emph{Spoiler wins the run} $(a_1,b_1),\ldots,(a_r,b_r)$.
\item The \emph{Duplicator wins the $r$-move Ehrenfeucht-Fra\"{i}ss\'e game on} $\bA$ and $\bB$ if the Duplicator can win every run of the game, i.e., if (s)he has a winning strategy for the Ehrenfeucht-Fra\"{i}ss\'e game. Otherwise,
the \emph{Spoiler wins the $r$-move Ehrenfeucht-Fra\"{i}ss\'e game}.
\item We write $\bA \sim_r B$ to denote that the Duplicator wins the $r$-move Ehrenfeucht-Fra\"{i}ss\'e game  on $\bA$ and $\bB$.
\end{itemize}
From this definition follows that $\sim_r$ is an equivalence relation on the class of all $\sigma$-structures. For a formal definition of the winning strategy for the Duplicator, see for example~\cite[Definition 3.4]{Kol}.
Ehrenfeucht-Fra\"{i}ss\'e games characterize definability in first-order logic. To describe this connection, we need the following definition.

Let $F$ be a first-order formula over a vocabulary $\sigma$. The \emph{quantifier rank} of $F$, denoted by $\qr(F)$, is defined inductively in the following way:
\begin{itemize}
\item if $F$ is atomic, then $\qr(F)=0$;
\item if $F$ is of the form $\neg F'$, then $qr(F)=\qr(F')$;
\item if $F$ is of the form $F'\vee F''$ or of the form $F'\wedge F''$, then $\qr(F)=\max\{ \qr(F'),\qr(F'')\}$;
\item if $F$ is of the form $\exists x F'$ or of the form $\forall x F'$, then $\qr(F)=\qr(F')+1$.
\end{itemize}

For a positive integer $r$ and two $\sigma$-structures $\bA$ and $\bB$, $A\equiv_r \bB$ denotes that $\bA$ and $\bB$ satisfy the same first-order sentences of quantifier rank $r$;
$\equiv_r$ is an equivalence relation on the class of all $\sigma$-structures. The main result of Ehrenfeucht and Fra\"{i}ss\'e asserts that the equivalence relations  $\equiv_r$ and $\sim_r$ coincide:

\begin{theorem}[\!\cite{Ehr,Fra}]\label{EhrFra}  Let $r$ be a positive integer and let $\bA$ and $\bB$ be two $\sigma$-structures. Then the following two conditions are equivalent:
\begin{itemize}
\item[(i)] $\bA\equiv_r \bB$, i.e., $\bA$ and $\bB$ have the same first-order sentences of quantifier rank $r$;
\item[(ii)] $\bA\sim_r \bB$, i.e., the Duplicator wins the $r$-move Ehrenfeucht-Fra\"{i}ss\'e game on $\bA$ and $\bB$.
\end{itemize}
Moreover, $\equiv_r$ has finitely many equivalence classes and each $\equiv_r$-equivalence class is definable by a first-order sentence of quantified rank $r$.
\end{theorem}

For a proof of this theorem, see~\cite{Kol}. A consequence of this theorem is the following result:

\begin{theorem}\label{EhrFra_bis} Let $\C$ be a class of $\sigma$-structures and $Q$ be a Boolean query on $\C$. Then the following statements are equivalent:
\begin{itemize}
\item[(a)] $Q$ is first-order definable on $\C$;
\item[(b)] there exists a positive integer $r$ such that, for every $\bA,\bB\in \C$, if $Q(\bA)=1$ and the Duplicator wins
the $r$-move Ehrenfeucht-Fra\"{i}ss\'e game on $\bA$ and $\bB$, then $Q(\bB)=1$.
\end{itemize}
\end{theorem}

This theorem provides the following method for studying first-order definability of Boolean queries on classes of $\sigma$-structures. Let $\C$ be a $\sigma$-structure and $Q$ be a Boolean query on $\C$.
To show that $Q$ is not first-order definable on $\C$, it suffices to show that for every positive integer $r$ there are $\bA_r,\bB_r\in \C$ such that
\begin{itemize}
\item $Q(\bA_r)=1$ and $Q(\bB_r)=0$;
\item the Duplicator wins the $r$-move Ehrenfeucht-Fra\"{i}ss\'e game on $\bA$ and $\bB$.
\end{itemize}
The method is also \emph{complete}, i.e., if $Q$ is not first-order definable on $\C$, then for every positive integer $r$ such structures $\bA_r$ and $\bB_r$ exist.

\subsection{What can be expressed and what cannot be expressed in FOL for graphs}\label{sec-folG}
Recall that an \emph{undirected graph} is a $\sigma$-structure $\bG=(V,E)$ with the universe
$V$ and the vocabulary $\sigma$ with one binary relation symbol $\E$ (interpreted as the edge relation) such that $\E(u,v)\Rightarrow \E(v,u)$ and $(\forall u) (\neg \E(u,u))$.
We start with a few queries on graphs, which are first-order definable:

\begin{example} Let $H=(V',E')$ be a graph with vertex-set $V'=\{ 1,\ldots,p\}$. The Boolean query ${\subgraph}_H$ meaning ``$\bG$ contains $H$ as an induced subgraph''
is definable by the first-order formula
\[
{\subgraph}_H \equiv (\exists v_1)(\exists v_2)\cdots (\exists v_p)(\bigwedge_{ij\in E'} \E(v_i,v_j) \wedge \bigwedge_{ij\notin E'} \neg \E(v_i,v_j)).
\]
This implies that the query $\subgraph_{H_1,\ldots,H_p}$  meaning ``$\bG$  contains at least one of the graphs $H_1,\ldots,H_p$ as an induced subgraph'' is also first-order definable.

Analogously, the binary query ``there exists a path of length $k$ from $x$ to $y$'' is definable by the first order formula
\[\varphi_k(x,y):=(\exists v_1)(\exists v_2)\cdots (\exists
  v_{k-1})(\E(x,v_1)\wedge \E(v_1,v_2)\wedge\cdots\wedge
  \E(v_{k-2},v_{k-1})\wedge \E(v_{k-1},y)).\]
Using the formulas $\varphi_k$, one can show that the binary query
$\dist_{\leq k}(x,y)$ meaning that ``the distance between $x$ and $y$
is at most $k$''  is definable by the first order formula
\[{\dist}_{\le k}(x,y) \equiv \varphi_k(x,y) \vee \varphi_{k-1}(x,y))
  \vee \cdots \vee \varphi_1(x,y) \vee (x = y). \]
The binary query $\dist_{k}(x,y)$ meaning that ``the distance between
$x$ and $y$ is at most $k$''  can then be defined as the first order formula
\[
  {\dist}_{k}(x,y)  \equiv {\dist}_{\leq k}(x,y) \wedge \neg {\dist}_{\leq
                     k -1} (x,y).
\]

Using the last queries, one can easily show that the Boolean queries $\isometric_H$ and $\isometric_{H_1,\ldots, H_p}$ meaning ``$\bG$ contains $H$ as an isometric subgraph'' and
``$\bG$ contains at least one  of the graphs $H_1,\ldots,H_p$ as an isometric subgraph'' are also first-order definable. For example,  if $H$ is a graph with the  vertex-set $\{ 1,\ldots,p\}$, then
$\isometric_H$ is definable by the formula
\[
{\isometric}_H\equiv (\exists v_1)(\exists v_2)\cdots (\exists v_p)\left(\bigwedge_{k=1}^p \left(\bigwedge_{\{i,j\}: d_H(i,j)=k} {\dist}_{k}(v_i,v_j)\right)\right).
\]
\end{example}

On the other hand, the most queries on graphs are not first-order definable, in particular the following well-known queries:
\begin{itemize}
\item The $\acyclicity$ query is the Boolean query such that $\acyclicity(\bG)=1$ iff $\bG$ is an acyclic graph;
\item The $\bipartite$ query is the Boolean query such that $\bipartite(\bG)=1$ iff $\bG$ is a bipartite graph;
\item The  $\connectivity$ query is the Boolean query such that $\connectivity(\bG)=1$ iff $G$ is a connected graph;
\item The $\even$ query  is the Boolean query such that $\even(\bG)=1$ iff $\bG$ has an even number of vertices.
\end{itemize}
All these results can be obtained via Ehrenfeucht-Fra\"{i}ss\'e games~\cite{Kol,Lib}.

\section{First Order Logic with Betweenness (FOLB)}\label{sec:FOLB}
In this section, we introduce the first-order logic with betweenness.
Betweenness was first formulated in geometry and nowadays has a rich history. Euclid, Pasch, Hilbert, Peano, and Tarski studied
betweeness in Euclidean geometry axiomatically. Menger~\cite{Me} and Blumenthal~\cite{Blu} investigated \emph{metric betweeness}, i.e., betweenness in general
metric spaces. Inspired by the work of Pasch,  Pitcher and Smiley~\cite{PiSm} and Sholander~\cite{Sh1,Sh2,Sh3} were the first to investigate betweenness
in the discrete setting: in lattices, partial orders, trees, and median semilattices. In graphs, the study of metric betweenness was initiated
by Mulder~\cite{Mu}. Prenowitz and Jantosciak~\cite{PrJa} were the first to investigate
the notion of betweenness in the setting of abstract convexity by introducing the concept of \emph{join space}. Hedlikov\'a represented the
betweenness relation as a ternary relation and introduced the concept of \emph{ternary space}; the betweenness relation in a ternary space
unifies the metric, order and lattice betweenness. Finally, this led
to the equivalent concept of \emph{geometric interval space}~\cite{VdV}.

\subsection{Betweenness and interval spaces}

Let $X$ be any finite set. For each pair $u,v$ of points in $X$, let $uv$ be a subset of $X$, called the \emph{interval} between $u$ and $v$.
Then $X$ is a \emph{(finite) interval space}~\cite{VdV} if and only if
\begin{enumerate}
\item[(I1)] $u\in uv$;
\item[(I2)] $uv=vu$;
\end{enumerate}
Every interval space gives rise to a betweenness relation: we will say that a point $x$ is \emph{between the points} $u$ and $v$ (notation $uxv$)
if $x\in uv$.
The interval space $X$ is said to be \emph{geometric}
if it satisfies the following three conditions for all $u,v,w,x\in X$~\cite{BaVdVVe,Ve}:
\begin{enumerate}
\item[(I3)] $uu=\{ u\}$,
\item[(I4)] $w\in uv$ implies $uw\subseteq uv$,
\item[(I5)] $v,w\in ux$ and $v\in uw$ implies $w\in vx$.
\end{enumerate}
A particular instance of geometric interval space is any metric space $(X,d)$: the intervals are the metric intervals $uv=\{ x\in X: d(u,x)+d(x,v)=d(u,v)\}$.

For each point $u$ one defines the \emph{base-point relation at $u$} as follows: $x\le_u y$ if and only if $x\in uy$. The next lemma summarizes an equivalent description of
geometric interval spaces~\cite[Section 27]{VdV}:

\begin{lemma}\label{geometric-interval-space} An interval space $X$ is geometric if and only if it satisfies the following conditions:
\begin{enumerate}
\item[(a)] $w\in ux$ and $x\in uw$ implies $w=x$;
\item[(b)] $v\in uw$ and $w\in ux$ implies $v\in ux$ and $w\in vx$;
\item[(c)] for each point $u$ the base-point relation $\le_u$ is a partial order such that for any $v\le_u w$ we have $vw=\{ x: v\le_u x\le_u w\}$.
\end{enumerate}
\end{lemma}

Let $X$ be any set  together with a ternary relation $uxv$. If $u,w,v\in X$ and $uwv$, then $w$ is said to be \emph{between} $u$ and $v$.
The \emph{interval} $uv$ is defined as the set of all $w\in X$ between $u$ and $v$. A \emph{ternary space} (which can be equally called a \emph{space with betweenness}) is a set $X$  together
with a ternary relation $uxv$ satisfying the following conditions~\cite{He}:
\begin{enumerate}
\item[(B1)] $uwv$ implies $vwu$;
\item[(B2)] $uwv$ and $uvw$ implies $v=w$;
\item[(B3)] $uwv$ and $uvx$ implies $wvx$ and $uwx$.
\end{enumerate}

From Lemma \ref{geometric-interval-space} it follows that a geometric interval space is exactly a ternary space satisfying the property that $u\in uv$ (i.e., $uuv$) for all $u,v\in X$.

An interval $uv$ of an interval space $X$ is called an \emph{edge} if $u\ne v$ and $uv=\{ u,v\}$; the edges then form the \emph{graph} $G(X)$ of the interval space $X$.

\begin{lemma}[\!\cite{BaCh_helly}]\label{discrete}  Let $X$ be a finite geometric interval space. Then the graph $G(X)$ of $X$ is connected.
\end{lemma}

The graph $G(X)$ of a finite geometric interval space $X$ can be regarded as a metric space, where the standard graph-metric $d$ accounts the lengths of shortest paths in the graph. We denote by
$I(u,v)=\{ x\in X: x \mbox{ on a shortest path between } u \mbox{ and } v\}$ the corresponding intervals in $G(X)$ which have to be distinguished from the intervals $uv$ in $X$. An interval space $X$ is called
\emph{graphic}~\cite{BaCh_helly,VdV} if the equality $uv=I(u,v)$ holds for all points $u,v$ of the space.

A simple sufficient condition for a finite interval space to be graphic was given in~\cite{BaCh_helly}. An interval space $X$ is said to satisfy the \emph{triangle condition} if for any three points $u,v,w$ in $X$ with

\begin{enumerate}
\item[(ITC)] $uv\cap uw=\{ u\}, uv\cap vw=\{ v\},$ and $uw\cap vw=\{ w\}$, the intervals $uv,uw,vw$ are edges whenever at least one of them is an edge.
\end{enumerate}

\begin{theorem}[\!\cite{BaCh_helly}]\label{graphic-triangle-condition}  A finite geometric interval space $X$ satisfying the axiom (ITC) is graphic.
\end{theorem}

Graphic interval spaces have been characterized by Mulder and Nebesk\'{y}~\cite{MuNe} (improving over the previous such characterizations obtained by Nebesk\'{y}). Additionally,
to axioms (I1)-(I5) of a geometric interval space, they require two
additional axioms introduced in~\cite{Ne}:

\begin{enumerate}
\item[(I6)] $uu'=\{ u,u'\}, vv'=\{ v,v'\}, u\in u'v'$, and $u',v'\in uv$ imply $v\in u'v'$;
\item[(I7)] $uu'=\{ u,u'\}, vv'=\{ v,v'\}, u'\in uv, v'\notin uv$, and $v\notin u'v'$ imply $u'\in uv'$.
\end{enumerate}

\begin{theorem}[\!\cite{MuNe}]\label{graphic-Mulder-Nebesky}  A finite geometric interval space $X$ is graphic if and only if it satisfies the axioms (I6) and (I7).
\end{theorem}

Observe that if $Y$ is a subset of an interval space $X$, we can
define an interval structure on $Y$ by taking the intersection of the
interval $uv$ in $X$ with $Y$ for any pair $u,v \in Y$. If $X$ is a
graphic interval space, then $Y$ endowed with this inherited interval
structure is also a graphic interval space. Note however that $G(Y)$
may be different from the subgraph of $G(X)$ induced by $Y$.

\subsection{FOLB for graphs}

Given a ternary predicate $B$ on a finite set $V$, we define the
binary predicate $\E_{\B}$ on $V$ as follows:
$\E_{\B}(u,v) := u \neq v \wedge \left( \B(u,x,v) \implies (x = u) \vee (x = v) \right)$.

A \emph{graphic interval structure} is a $\sigma$-structure $(V,B)$
where $V$ is a finite set and $B$ is a ternary predicate on $V$ satisfying the following axioms:

\begin{enumerate}[{(IB}1)]
\item $\forall u \forall v  \B(u,u,v) $
\item $\forall u \forall v \forall x  \B(u,x,v) \implies \B(v,x,u)$
\item $\forall u \forall x \B(u,x,u) \implies x=u $
\item $\forall u \forall v \forall w \forall x  \B(u,w,v) \wedge \B(u,x,w) \implies \B(u,x,v)$
\item
  $\forall u \forall v \forall w \forall x \B(u,v,x) \wedge \B(u,w,x) \wedge \B(u,v,w)
  \implies \B(v,w,x)$
\item $\forall u \forall u' \forall v \forall v' \E_{\B}(u,u') \wedge \E_{\B}(v,v') \wedge \B(u',u,v') \wedge \B(u,u',v) \wedge \B(u,v',v) \implies \B(u',v,v')$
\item $\forall u \forall u' \forall v \forall v' \E_{\B}(u,u') \wedge \E_{\B}(v,v') \wedge \B(u,u',v) \wedge \neg \B(u,v',v) \wedge \neg \B(u',v,v') \implies \B(u,u',v')$.
\end{enumerate}

Observe that since $B$ satisfies (IB2), $G_{\B} = (V,E_{\B})$ is an
undirected graph seen as a $\sigma'$-structure (as defined in
Section~\ref{sec-folG}).  By Lemma~\ref{discrete}, $G_{\B}$ is a
connected graph.  Since $B$ satisfies (IB1)--(IB7), by
Theorem~\ref{graphic-Mulder-Nebesky}, for any $u,v,x \in V$,
$x \in I_{G_{\B}}(u,v)$ if and only if $\B(u,x,v)$. When $\B(u,x,v)$ is
true, it means that $x$ belongs to the interval $I_{G_{\B}}(u,v)$.

When considering the class $\C$ of $\sigma$-structures $(V,B)$
satisfying axioms (IB1)--(IB7), we say that a query $Q$ on $\C$ is
definable in first order logic with betweeness (\emph{FOLB-definable})
if it can be defined by a first order formula $F$ over $(V,B)$.

Observe that by the definition of $\E_{\B}$, any FOL-definable query is
also FOLB-definable. In particular the queries $\dist_k$ and
$\dist_{\leq k}$ are FOLB-definable.

\subsection{What can be expressed in FOLB for graphs: first results}

There are properties in FOLB that cannot be expressed using only
FOL. Namely, we prove that $\bipartite$ and $\tree$ are FOLB-definable,
where $\tree$ is $\acyclicity \wedge \connectivity$. Since in FOLB we consider only
connected graphs, $\connectivity$ is a trivial query in FOLB.

A graph is bipartite if and only if for any edge $uv$ and any vertex
$x$, the distances from $x$ to $u$ and to $v$ are different, and thus
if and only if either $u \in I_{G}(x,v)$ or $v \in
I_{G}(x,u)$. Consequently, $\bipartite$ is definable by the following
FOLB-formula:
\[
  \bipartite \equiv \left(\forall u \forall v \forall x \E_{\B}(u,v) \implies \B(x,u,v) \vee \B(x,v,u)\right).
\]

A tree is bipartite.  In a bipartite connected graph, if $G$ is not a
tree, there are two vertices $u, x$ such that $u$ has two neighbors in
the interval $I_G(u,x)$. Indeed, consider a cycle $C$ and an arbitrary
vertex $x$. Let $u$ be the vertex of $C$ that is the furthest from
$x$. Since $G$ is bipartite, the two neighbors $v,w$ of $x$ on $C$
belong to the interval $I(x,u)$.  Consequently, $\tree$ is definable by
the following FOLB-formula:
\[
  \tree \equiv \left(\bipartite \wedge \forall u \forall v \forall w
    \forall x \E_{\B}(u,v) \wedge
  \E_{\B}(u,w) \wedge \B(x,v,u) \wedge \B(x,w,u) \implies v=w\right).
\]

We will use  the predicates $\tri(x,y,z)$, $\squ(x,y,z,u),$ and $\pentagon(x,y,z,u,v)$,
which are true if and only if the vertices $x,y,z$, $x,y,z,u$, and $x,y,z,u,v$ induce respectively
a triangle, a square, or a pentagon of a graph $G$. We will also use the predicate $\ii(v,x,y)$,
which is true if and only of the intervals $I(x,v)$ and $I(y,v)$ intersects only in the vertex $v$.
$\ii(v,x,y)$ can be written as the FOLB-formula $(\forall v' \B(x,v',v) \wedge \B(y,v',v) \implies v'=v)$.

Given four vertices $u, v, x, y$ of $G$, the following predicate
express that $x$ and $y$ belong to a common shortest path going from $u$ to $v$ (reaching first $x$ and then $y$):
\[
 \colinear(u,x,y,v) \equiv \B(u,x,v) \wedge \B(x,y,v)
\]

Three vertices $x, y, z$ of a graph $G$ define a \emph{metric
  triangle} $xyz$~\cite{Ch_metric} if $I(x,y) \cap I(x,z) = \{x\}$,
$I(x,y) \cap I(y,z) = \{y\}$, and $I(x,z) \cap I(y,z) = \{z\}$. This
can be expressed using the predicate
\[
  \mtriangle(x,y,z) \equiv \ii(x,y,z) \wedge \ii(y,x,z) \wedge
  \ii(z,x,y).
\]
The size of a metric triangle $xyz$ is
$\max(d(x,y), d(x,z), d(y,z))$.

Given three vertices $x, y, z$, a metric triangle $x'y'z'$ is a
\emph{quasi-median} of $x, y, z$ if
$d(x,y) = d(x,x') + d(x',y') + d(y',y)$,
$d(x,z) = d(x,x') + d(x',z') + d(z',z)$, and
$d(y,z) = d(y,y') + d(y',z') + d(z',z)$, For any vertices $x, y, z$,
one can obtain a quasi-median $x'y'z'$ of $x, y, z$ by taking
$x' \in I(x,y) \cap I(x,z)$ furthest from $x$,
$y' \in I(x',y) \cap I(y,z)$ furthest from $y$, and
$z' \in I(x',z) \cap I(y',z)$ furthest from $z$.  It can be expressed
by the following predicate:
\begin{align*}
  \quasimed(x,y,z,x',y',z') \equiv  &\mtriangle(x',y',z') \wedge \colinear(x,x',y',y) \\
                                    & \wedge \colinear(x,x',z',z) \wedge \colinear(y,y',z',z).
\end{align*}

A quasi-median $x'y'z'$ of $x, y, z$ such that $x' = y' = z'$ is
called a \emph{median} of $x, y, z$. Equivalently, $m$ belongs to
$I(x,y)\cap I(x,z)\cap I(y,z)$.
\begin{align*}
  \med(x,y,z,m) & \equiv \quasimed(x,y,z,m,m,m)\\
  & \equiv  \B(x,m,y) \wedge \B(x,m,z) \wedge \B(y,m,z).
\end{align*}

The following metric conditions on a graph $G$ play an important way in the definition of many graph classes:
\begin{itemize}
 \item \emph{Triangle Condition} ($\TC$): for any three vertices $v,x,y$ such that $d(v,x)=d(v,y)$ and $x\sim y$, there exists a vertex  $z\in I(x,v)\cap I(y,v)$ such that $xzy$ is a triangle of $G$;
 \item\emph{Quadrangle Condition} ($\QC$): for any four vertices $v,x,y,u$ such that $d(v,x)=d(v,y)=d(v,u)-1$ and $u\sim x,y$, $x\nsim y$, there exists
 a vertex $z\in I(x,v)\cap I(y,v)$ such that $xzyu$ is a square of $G$;
\item \emph{Triangle-Pentagon Condition} ($\TPC$): for any three vertices $v,x,y$ such that $d(v,x)=d(v,y)$ and $x\sim y$, either there exists a vertex  $z\in I(x,v)\cap I(y,v)$ such that $xzy$ is a triangle of $G$, or there exist vertices $z,x',y'$ such that $xx'zy'y$ is a pentagon of $G$, $z\in I(x,v)\cap I(y,v)$, and $d(x,z)=d(y,z)=2$;
 \item \emph{Interval Neighborhood Condition} ($\INC$): for any two distinct vertices $u,v\in V$, the neighbors of $u$ in $I(u,v)$ form a clique.
\end{itemize}
We denote the respective queries by $\tc$, $\qc$, $\pc$, and $\inc$  and we show that these properties are FOLB-definable:

\begin{align*}
  \tc   & \equiv \forall v \forall x \forall y \E_{\B}(x,y) \wedge \ii(v,x,y) \implies
                \tri(x,y,v) \\
  \qc  & \equiv  \forall v \forall x \forall y \forall u \E_{\B}(u,x) \wedge \E_{\B}(u,y) \wedge
                 \neg \E_{\B}(x,y) \\
              & \quad \wedge \B(u,x,v)
                \wedge \B(x,u,y) \wedge \B(u,y,v) \wedge \ii(v,x,y)\\
               & \quad \implies \squ(u,x,v,y) \\
\tpc   &\equiv \forall v \forall x \forall y \E_{\B}(x,y) \wedge \ii(v,x,y) \implies \tri(x,y,v) \\
  & \quad \vee \exists x' \exists y' \pentagon(x,x',v,y',y) \\
  \inc   &\equiv  \forall u \forall v \forall x \forall y \E_{\B}(u,x) \wedge \E_{\B}(u,y) \wedge
                 \B(u,x,v) \wedge \B(u,y,v) \\
  & \quad \implies \E_{\B}(x,y).\\
\end{align*}

Observe that if a graph $G$ satisfies $\INC$, then $G$ does not
contain any square. Moreover, when $G$ satisfies $\QC$, $G$ satisfies
$\INC$ if and only if $G$ does not contain any square.

\section{Weakly modular graphs,  their subclasses and superclasses}\label{sec:weakly-modular}
In this section, we present the FOLB-definability of weakly modular graphs and their main subclasses and super-classes, which constitute an important part of Metric Graph Theory. Subclasses of weakly modular graphs are the following classes of graphs:  median, modular, quasi-modular, quasi-median, pseudo-modular, weakly median, bridged and weakly bridged, Helly, dually polar, and sweakly modular. Meshed graphs constitute a super-class of weakly modular graphs. Basis graphs of matroids and of even $\Delta$-matroids are subclasses of meshed graphs.

\subsection{Weakly modular graphs}

  Weakly modular graphs have been introduced in the papers~\cite{Ch_metric} and~\cite{BaCh_helly}. A nice local-to-global theory of weakly modular graphs and their subclasses mentioned above has been developed in the recent paper~\cite{CCHO}. For results about weakly modular graphs, the reader can consult the survey~\cite{BaCh_survey} and the paper~\cite{CCHO}.

A graph is \emph{weakly modular} if it satisfies the triangle condition ($\TC$) and the quadrangle condition ($\QC$). Thus being weakly modular can be expressed by the following FOLB-query:
\[
   \wmodular \equiv  \tc \wedge \qc.
\]

In~\cite{Ch_metric}, weakly modular graphs have been characterized as
graphs in which metric triangles $xyz$ are \emph{strongly
  equilateral}, i.e., for any $u \in I(y,z)$, we have
$d(x,u) = d(x,y) = d(x,z)$. In fact, this characterization leads to another FOLB query characterizing weakly modular graphs:

\begin{align*}
  \stronglyequilateral(x,y,z) & \equiv
  \forall u \forall v \E_{\B}(u,v) \implies \\
  & \quad \big( ( (\B(y,u,z) \wedge \B(y,v,z) ) \implies \neg (\B(x,u,v) \vee \B(x,v,u))) \\
  & \quad \wedge ( (\B(x,u,z) \wedge \B(x,v,z) ) \implies \neg (\B(y,u,v) \vee \B(y,v,u))) \\
  & \quad \wedge ( (\B(x,u,y) \wedge \B(x,v,y) ) \implies \neg (\B(z,u,v) \vee \B(z,v,u)) ) \big) \\
  \strongEquil & \equiv  \forall x\forall y \forall z \stronglyequilateral(x,y,z)\\
\end{align*}

The predicate $\stronglyequilateral(x,y,z)$ establishes that for any
adjacent vertices $u, v$ of $I(y,z)$ (for $I(x,y)$ and $I(x,z)$, the
arguments are similar), we have $u \notin I(x,v)$ and
$v \notin I(x,u)$, yielding $d(x,u) = d(x,v)$. Consequently, the
connectedness of $I(y,z)$ establishes that all vertices of $I(y,z)$
have the same distance to $x$. This shows that
$\stronglyequilateral(x,y,z)$ is true if and only if $xyz$ is a
strongly equilateral metric triangle.

 {$$ \wmodular \equiv \strongEquil.$$ }

We say that a graph $G$ has equilateral metric triangles if every
metric triangle $xyz$ of $G$ is \emph{equilateral}, i.e.,
$d(x,y) = d(x,z) = d(y,z)$. One can ask if graphs with equilateral
metric triangles are FOLB-definable.

A \emph{modular} graph is a bipartite weakly modular graph, i.e., a bipartite graph satisfying the quadrangle condition. Thus being  modular can be expressed by the following FOLB-queries:
\begin{align*}
  \modular &\equiv \bipartite \wedge \wmodular\\
           &\equiv \bipartite \wedge \qc.\
\end{align*}

A graph is \emph{pseudo-modular}~\cite{BaMu_pm} if it satisfies the triangle
condition and if for any three vertices $v,x,y$ such that $d(x,y) = 2$
and $d(v,x)=d(v,y)=k$, there exists a vertex $z\sim x,y$ such that
$d(v,z)=k-1$. The second property can be viewed as a strengthening of
the quadrangle condition. In fact, a graph is pseudo-modular if and
only if all metric triangles have size at most $1$ and can thus be
FOLB-defined by the following formula:
\[
  \pmodular \equiv \forall x \forall y \forall z \mtriangle(x,y,z) \implies (x = y = z)
  \vee \tri(x,y,z).
\]

A \emph{quasi-modular} graph is a $K_4^-$-free weakly modular graph~\cite{BaMuWi} and thus
being quasi-modular is a FOLB-definable property.
\[
  \qmodular \equiv \wmodular \wedge \neg {\subgraph}_{K_4^-}.
\]
Quasi-modular graphs are pseudo-modular but the converse inclusion
does not hold.

A graph $G=(V,E)$ is called \emph{meshed}~\cite{BaCh_survey} if the
following condition ($\QC^-$) is satisfied for any three vertices
$v,x,y$ with $d(x,y)=2$: there exists a common neighbor $z$ of $x$ and
$y$ such that $2d(v,z)\le d(v,x)+d(v,y)$. This condition seems to be a
relaxation, but it implies the triangle condition (that is not implied
by the quadrangle condition). Conversely, ($\TC$) and ($\QC$) imply
($\QC^-$) and thus weakly modular graphs are meshed. In meshed graphs,
any metric triangle $xyz$ is equilateral~\cite{BaCh_conmed}.

\begin{lemma}
  $G$ is a meshed graph if and only if for any metric triangle
  $vxy$, if $d(x,y) = 2$, then $d(v,x) = d(v,y) =2$ and there exists
  $z \sim x,y$ such that $d(v,z) = 2$.
\end{lemma}

\begin{proof}
  Let $G$ be a meshed graph and consider a metric triangle $vxy$ such
  that $d(x,y)= 2$. Since metric triangles in meshed graphs are
  equilateral, we have $d(v,x)=d(v,y) = 2$, and by ($\QC^-$), there
  exists $z \sim x,y$ such that $d(v,z) \leq 2$. If $d(v,z) =1$, then
  $z \in I(v,x) \cap I(v,y) \cap I(x,y)$ and thus $vxy$ is not a
  metric triangle. Consequently, $d(v,z) = 2$.

  Consider now a graph $G$ such that for any metric triangle $vxy$ of
  $G$ with $d(x,y) = 2$, we have $d(v,x) = d(v,y) =2$ and there exists
  $z \sim x,y$ such that $d(v,z) = 2$. We show that ($\QC^-$) holds in
  $G$.  Consider three vertices $v,x,y$ such that $d(x,y) = 2$ and let
  $v'x'y'$ be a metric triangle such that $v',x' \in I(v,x)$,
  $v',y' \in I(v,y)$, and $x',y' \in I(x,y)$. If $x \neq x'$, then
  either $x'=y$ (if $y \in I(x,v)$) or $x' \sim x,y$. In the first
  case, let $z$ be a common neighbor of $x$ and $y$ and in the second
  case, let $z = x'$. In both cases, observe that
  $d(v,z) = d(v,x)-1 \leq d(v,y)+1$ and thus ($\QC^-$) holds for
  $v,x,y$. We can thus assume now that $x=x'$ and for similar reasons
  that $y = y'$. Since $d(x,y)=2$, we know by hypothesis that
  $d(v',x) = d(v',y) = 2$ and that there exists $z \sim x,y$ such that
  $d(v',z) = 2$. Consequently,
  $d(v,z) \leq d(v,v')+d(v',z) = d(v,v')+ 2 = d(v,v') + d(v',x) =
  d(v,x) = d(v,y)$ and thus ($\QC^-$) holds for $v,x,y$.
\end{proof}

Consequently, meshedness of a graph can be written as the following
FOLB query:
\begin{align*}
  \meshed  &\equiv \forall v \forall x \forall y \mtriangle(v,x,y) \wedge {\dist}_2(x,y) \implies {\dist}_2(x,v) \wedge {\dist}_2(y,v)\\
  & \quad \wedge (\exists z \E_{\B}(x,z) \wedge \E_{\B}(y,z) \wedge {\dist}_2(z,v)).
\end{align*}

The previous lemma establishes that meshed graphs are precisely the
graphs in which every metric triangle $xyz$ of size $2$, there exists
a common neighbor of $y$ and $z$ at distance $2$ from $x$.  Therefore
one can ask whether meshed graphs are exactly the graph where for each
metric triangle $xyz$, $y$ and $z$ can be connected by a shortest path
in which all vertices have the same distance to $x$.

\subsection{Median graphs} Median graphs constitute the most important class of graphs in Metric Graph Theory. This is due to the occurrence of
median graphs in completely different areas of mathematics and computer science. This is also due to their deep and rich combinatorial
and geometric structure, which was an inspiration for most of subsequent generalizations. Median graphs  originally arise in
universal algebra~\cite{Av,BiKi} and their properties have been first investigated in~\cite{Mu,Ne_med}. It was shown in~\cite{Ch_CAT,Ro} that the
cube complexes of median graphs are exactly the CAT(0) cube complexes, i.e., cube complexes of global non-positive curvature. CAT(0) cube complexes,
introduced and nicely characterized in~\cite{Gr} in a local-to-global way, are now one of the principal objects of investigation in geometric
group theory~\cite{Sa_survey}. Median graphs also occur in Computer Science: by~\cite{BaCo} they are exactly the domains of event structures
(one of the basic abstract models of concurrency)~\cite{NiPlWi} and median-closed subsets of hypercubes are exactly the solution sets
of 2-SAT formulas~\cite{MuSch,Sch}. The bijections between median graphs, CAT(0) cube complexes, and event structures have been used
in~\cite{ChCh_reg,ChCh_MSO,Ch_nice} to disprove three conjectures in
concurrency theory.  Finally, median graphs, viewed
as median closures of sets of vertices of a hypercube, contain all most  parsimonious (Steiner) trees~\cite{BaFoRo} and as such have been
extensively applied in human genetics. For a survey on median graphs and their connections with other discrete and geometric structures, see the
books~\cite{HaImKl,Knuth}, the surveys~\cite{BaCh_survey,KlMu}, and the paper~\cite{CCHO}.

A graph $G$ is \emph{median} if any triplet of vertices has a unique
median. Notice that, in modular graphs, any triplet of vertices
has at least one median, and thus median graphs are the modular graphs
where the medians are unique. Equivalently, median graphs are modular graphs not containing
induced $K_{2,3}$~\cite{Mu}.

\begin{align*}
  \modular & \equiv \forall x \forall  y \forall z \exists m \med(x,y,z,m) \\
           & \equiv \forall x \forall y \forall z \mtriangle(x,y,z) \implies (x=y=z)\\
  \median  & \equiv \forall x \forall y \forall  z \exists ! m \med(x,y,z,m) \\
           & \equiv \modular \wedge \neg {\subgraph}_{K_{2,3}}.
\end{align*}

Quasi-median graphs has been introduced and studied in~\cite{BaMuWi} and pseudo-median graphs has been introduced in~\cite{BaMu_pmedg}. For applications of quasi-median graphs in geometric group theory, see the PhD thesis~\cite{Ge}. Weakly median graphs has been introduced and characterized in~\cite{Ch_metric,BaCh_wmg}.  For results and bibliographic references about quasi-median and weakly median graphs, see the survey~\cite{BaCh_survey}.

A \emph{quasi-median} (respectively, \emph{pseudo-median},
\emph{weakly median}) graph is a quasi-modular (respectively,
{pseudo-modular}, {weakly modular}) graph in which each triplet of
vertices has a unique quasi-median. These definitions already show
that quasi-median, pseudo-median, and weakly median graphs can be
expressed by FOLB-queries. These classes of graphs have also been
characterized within their superclasses by forbidden
subgraphs (see Figure \ref{Forbidden_subgraphs}). Therefore, similarly to median graphs, we can describe
these classes by the following FOLB-queries.

\begin{align*}
\qmedian & \equiv \qmodular \wedge \neg {\subgraph}_{K_{2,3}}\\
  \pmedian & \equiv \pmodular \wedge \neg {\subgraph}_{H_1,H_2,H_3,H_4} \\
  \wmedian & \equiv \wmodular \wedge \neg {\subgraph}_{H_1,H_2,H_3,H_4}.
\end{align*}

\begin{center}
\begin{figure}
\includegraphics[width=150mm]{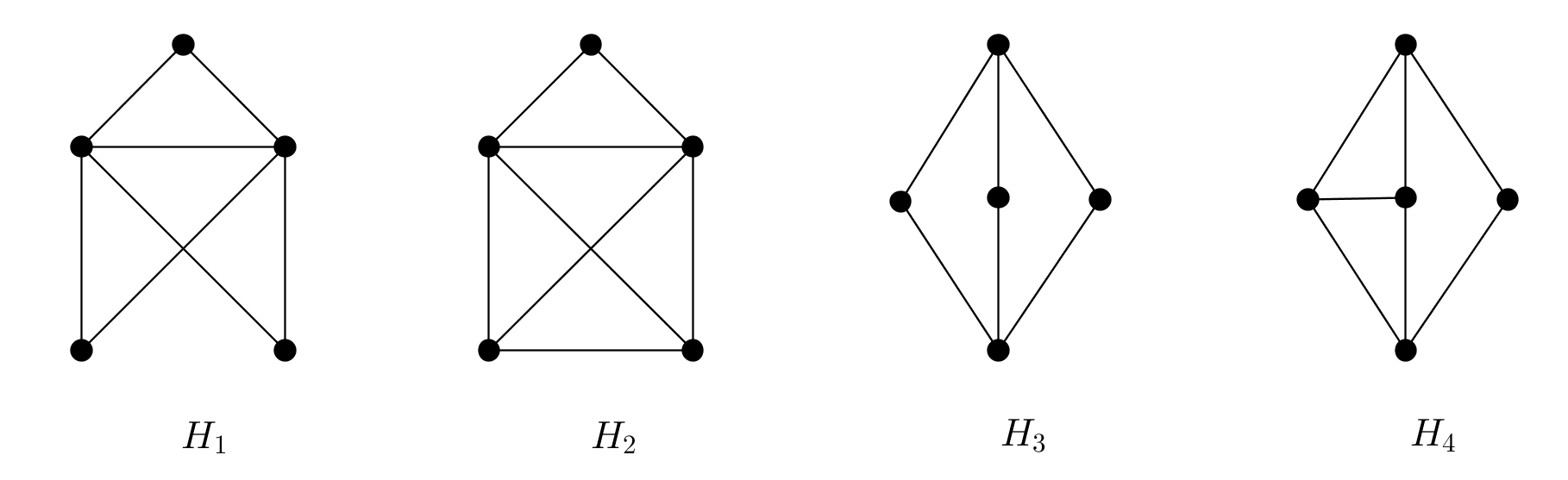}
\caption{Forbidden subgraphs }
\label{Forbidden_subgraphs}
\end{figure}
\end{center}

\subsection{Bridged graphs and graphs with convex balls}

The convexity of balls and the convexity of neighborhoods of convex sets are fundamental features of
geodesic metric spaces, which are globally nonpositively curved~\cite{BrHa}. CAT(0) (alias nonpositively curved geodesic metric spaces), introduced by Gromov in his seminal paper~\cite{Gr}, and groups acting on them are fundamental objects of study in metric geometry and geometric group theory. The graphs  with convex
balls have been introduced and characterized in~\cite{FaJa,SoCh} as graphs without embedded isometric cycles of lengths different from 3 and 5 and in which all neighbors of a
vertex $u$ on shortest $(u,v)$-paths form a clique.  One of their important subclass  is the class of bridged
graphs: these are the graphs without embedded isometric cycles of length greater than 3 and they are exactly the graphs in which
the balls around convex sets are convex~\cite{FaJa,SoCh}.  It was proved in~\cite{Ch_CAT} that the simplicial complexes having bridged graphs as 1-skeletons are the simply connected simplicial complexes in which the links of vertices are flag complexes without embedded 4- and 5-cycles. Under  this form, bridged graphs and complexes have been rediscovered in~\cite{JaSw} under the name ``systolic complexes'' and have been investigated in geometric group theory as the good combinatorial analogs of CAT(0) metric spaces.
Weakly systolic (alias weakly bridged) graphs and complexes have been introduced in~\cite{Os} and further studied in~\cite{ChOs}, where it is shown that they are exactly the weakly modular graphs with convex balls. More recently, a detailed investigation of graphs with convex balls and of their triangle-pentagon complexes has been provided in~\cite{ChChGi}.

A graph $G$ is a \emph{graph with convex balls} if all balls of $G$ are convex sets. A graph $G$ is called \emph{bridged} if all neighborhoods of convex sets are convex.
Bridged graphs are exactly weakly modular graphs without $C_4$ and
$C_5$ as induced subgraphs~\cite{Ch_metric}. \emph{Weakly bridged graphs} are the weakly
modular graphs with convex balls; they are exactly the weakly modular
graphs without $C_4$~\cite{ChOs}. More recently, it was shown in~\cite{ChChGi}
that graphs with convex balls are exactly the graphs satisfying the
Triangle-Pentagon and the Interval Neighborhood Conditions.

We can thus describe these classes by the following FOLB-queries:

\begin{align*}
  \bridged  & \equiv \wmodular \wedge \neg {\subgraph}_{C_4,C_5}\\
  \wbridged & \equiv \wmodular \wedge \neg {\subgraph}_{C_4} \\
            & \equiv \wmodular \wedge \inc \\
  \cballs   & \equiv \tpc \wedge \inc.
\end{align*}

Bucolic graphs have been introduced and studied in~\cite{BCCGO} and they are the graphs that can be obtained by retractions from
Cartesian products of weakly bridged graphs.  This is a far reaching common
generalization of median graphs and (weakly) bridged graphs.
Notice that median graphs are exactly the graphs which can be obtained
from cubes (Cartesian products of edges) via gated amalgamations. The bucolic graphs are exactly
the graphs which can be obtained from Cartesian products of weakly bridged graphs by gated amalgamations.
It was shown in~\cite{BCCGO} that prism complexes of bucolic graphs have many properties of CAT(0) spaces.

\emph{Bucolic graphs} have been characterized in~\cite{BCCGO} as weakly
modular graphs without $K_{2,3}$, $W_4$, and $W_4^-$ as induced
subgraphs. Consequently, they can be characterized by the following FOLB-query:
\[
  \bucolic \equiv \wmodular \wedge \neg {\subgraph}_{K_{2,3},W_4,W_4^-}.
\]

\subsection{Helly graphs} A geodesic metric space is injective if any family of
pairwise intersecting balls has a non-empty intersection. Similarly to CAT(0) spaces, injective metric spaces (called also hyperconvex spaces or absolute retracts) appear
independently in various fields of mathematics and computer science: in topology and metric
geometry, in combinatorics, in functional
analysis and fixed point theory, and in algorithmics.  The distinguishing feature of
injective spaces is that any metric space admits an injective hull, i.e., the smallest injective
space into which the input space isometrically embeds~\cite{Dr,Is}. Helly graphs are a discrete counterpart
of injective metric spaces and, again, there are many equivalent definitions of such graphs, hence they are also
known as e.g. absolute retracts~\cite{BaPe,BaPr,Qu,Pe}. More recently, a local-to-global characterization of
Helly graphs has been given in~\cite{CCHO} and the groups acting on Helly graphs have been investigated in~\cite{CCHGO}.

A family of subsets $\mathcal F$ of a set $X$ satisfies the
\emph{Helly property} if for any subfamily ${\mathcal F}'$ of
$\mathcal F$, the intersection
$\bigcap {\mathcal F}'=\bigcap \{ F: F\in {\mathcal F}'\}$ is nonempty
if and only if $F\cap F'\ne \emptyset$ for any pair
$F,F'\in {\mathcal F}'$.

A graph $G$ is a \emph{Helly graph} if the family of balls of $G$ has
the Helly property.  A \emph{clique-Helly graph} is a graph in which
the collection of maximal cliques has the Helly property. Observe that
each Helly graph is clique-Helly but the converse does not hold:
indeed, clique-Hellyness is a local property while Hellyness is a
global property. It was shown in~\cite{CCHO} that $G$ is Helly if and
only if $G$ is clique-Helly and the clique complex $X(G)$ of a graph
$G$ is simply connected.  The following result also follows from~\cite{CCHO}:

\begin{lemma}
  A graph $G$ is Helly if and only if $G$ is a clique-Helly weakly
  modular graph in which any $C_4$ is included in a $W_4$.
\end{lemma}

Clique-Hellyness of a graph can be characterized by the following
condition~\cite{Dra,Szw} that is a particular case of the
Berge-Duchet condition~\cite{BeDu}.  A graph $G$ is clique-Helly if and
only if for any triangle $T$ of $G$ the set $T^*$ of all vertices of
$G$ adjacent with at least two vertices of $T$ contains a vertex
adjacent to all remaining vertices of $T^*$.

Since it is easy to express the fact that any square is included in a 4-wheel, we can describe Helly and clique-Helly graphs by the following FOLB-queries.
\begin{align*}
  \cqwq & \equiv \forall w \forall x \forall y  \forall z \squ(w,x,y,z) \implies \exists u \left(\E_{\B}(w,u) \wedge \E_{\B}(x,u) \wedge \E_{\B}(y,u) \wedge \E_{\B}(z,u)\right) \\
  \chelly & \equiv \forall x \forall y \forall z \tri(x,y,z)\implies \big(\exists u \forall v \big( v \neq u \wedge (\tri(x,y,v) \vee \tri (x,z,v)  \\
            &\quad \vee \tri(y,z,v))\big) \implies \E_{\B}(u,v)\big)\\
  \helly & \equiv \chelly \wedge \wmodular \wedge \cqwq.
\end{align*}

\subsection{Dual polar graphs,  sweakly  modular graphs, and basis graphs}
Projective and polar spaces are the most fundamental types of incidence geometries; for definition and for a full account
 of their theory, see~\cite{Shu,Ti}. Dual polar spaces are duals of polar spaces, seen as incidence
geometries. It was observed in~\cite{BaCh_survey} that \emph{dual polar graphs}, which are the collinearity graphs
of dual polar spaces, are weakly modular. This is a simple
consequence of Cameron’s characterization~\cite{Ca} of dual polar graphs (which is of metric type). Notice
also that there is a local-to-global characterization of dual polar graphs, due to
Brouwer and Cohen~\cite{BrCo}. Moreover, it was shown in~\cite{CCHO} that dual polar
graphs are exactly the thick weakly modular graphs without $K_{4}^-$ and $K_{3,3}^-$ as isometric subgraphs.

The \emph{swm-graphs (sweakly modular graphs)} represent a natural
extension of dual polar graphs, because they are defined as the weakly
modular graphs not containing induced $K^{-}_4$ and isometric
$K^{-}_{3,3}$~\cite{CCHO}. They also extend the \emph{strongly modular
  graphs} of~\cite{Hi}, that are exactly the modular graphs not
containing induced $K^{-}_4$ and isometric $K^{-}_{3,3}$. Strongly
modular graphs arise in the dichotomy theorem of~\cite{Hi} as exactly
the graphs for which the multifacility location problem (alias the
0-extension problem) is polynomial.  One can define a cell structure
on sweakly modular graphs, where cells are orthoschemes corresponding
to gated dual polar graphs. In~\cite{CCHO}, various characterization
and properties of cell complexes arising from swm-graphs are given.

Due to the characterizations of dual polar graphs and swm-graphs given above, they can be described by the following FOLB-queries:
\begin{align*}
  \thick & \equiv \forall u \forall v {\dist}_2(u,v) \implies \exists x \exists y \squ(u,x,v,y)\\
  \dualpolar &\equiv \wmodular \wedge \thick \wedge \neg {\isometric}_{K_4^-,K_{3,3}^-} \\
  \stmodular & \equiv \modular  \wedge \neg {\isometric}_{K_4^-,K_{3,3}^-}\\
  \swmodular & \equiv \wmodular  \wedge \neg {\isometric}_{K_4^-,K_{3,3}^-}.
\end{align*}

Matroids constitute an important unifying structure in combinatorics, geometry, algorithmics, and
combinatorial optimization~\cite{Ox}. Matroids can be defined in several equivalent ways, in particular
in terms of bases. Basis graphs of matroids have the bases as vertex-sets and elementary exchanges as edges.
Basis graphs faithfully represent their matroids, thus studying the basis graph amounts
to studying the matroid itself. Basis graphs are also know to coincide with the 1-skeleton of the basis
polytope of a matroid. Maurer~\cite{Mau} characterized basis graphs of matroids in a pretty way by involving
two local conditions and a global metric condition. He also conjectured that the metric conjecture can be replaced by a local
condition and simple connectivity of its triangle-square complex. This conjecture was confirmed in~\cite{ChChOs}.
Delta-matroids constitute an interesting generalization of matroids. They have been introduced in different but independent
way in the papers~\cite{Bou,ChKa} and~\cite{DrHa} and later have been considered as the Lagrangian matroids in the context of
Coxeter matroids~\cite{BoGeSe}. A characterization of basis graphs of even Delta-matroids in the spirit of Maurer's characterization
was given in~\cite{Ch_basis}. Using it, a local-to-global characterization was given in~\cite{ChChOs}.

A \emph{matroid} on a finite set of elements $I$ is a collection $\mathcal B$ of subsets of $I$, called bases, which satisfy the following exchange
axiom: for all $A,B\in \mathcal B$ and $a\in A\setminus B$ there exists $b\in B\setminus A$ such that $A\setminus \{a\}\cup \{ b\}\in \mathcal B$
(the base $A\setminus \{a\}\cup \{ b\}$ is obtained from the base $A$ by an \emph{elementary exchange}). The basis
graph $G =G({\mathcal B})$ of a matroid $\mathcal B$ is the graph whose vertices are the bases of $\mathcal B$ and edges are
the pairs $A,B$ of bases differing by an elementary exchange, i.e., $|A\Delta B|= 2$. \emph{Delta-matroids} are defined in the same way as the matroids by replacing
everywhere in the basis exchange axiom the set difference by the symmetric difference. An even Delta-matroid is a delta-matroid in which all bases have an even number of elements.
The basis graph of an even Delta-matroid is defined in the same way as the basis graph of a matroid. The basis graphs of matroids are isometric
subgraphs of Johnson graphs (because all bases have the same cardinality) and the basis graphs of even Delta-matroids are isometric subgraphs of
halved-cubes (because all bases have even size).

By Maurer's~\cite{Mau} characterization and its improvement provided in~\cite{ChChOs}, a graph is the basis graph of a matroid if and only if it satisfies
the two following conditions:
\begin{itemize}
\item\emph{Positioning Condition} ($\PoC$): for any square $wxyz$ and
  any vertex $v$, $d(v,w)+d(v,y) = d(v,x)+d(v,z)$.
\item\emph{$2$-Interval Condition $3$} ($\TiCT$): for any two vertices
  $u,v$ at distance $2$, $I(u,v)$ induces a square, a pyramid, or a
  $3$-octahedron.
\end{itemize}
In fact, Maurer's characterization also uses the \emph{link condition} that  the neighborhood of each vertex
induces the line graph of a bipartite graph, however he conjectured that this condition is redundant and this was confirmed
in~\cite{ChChOs}.

Observe that the positioning condition can be restated as follows: for
any square $S=wxyz$ and any vertex $v$, either all vertices of $S$ are
at the same distance $k$ to $v$, or two adjacent vertices of $S$ are
at distance $k$ from $v$  and the two other adjacent vertices are at
distance $k+1$ of $v$, or one vertex is at distance $k+1$ of $v$, its
neighbors are at distance $k$ of $v$, and the fourth vertex is at
distance $k-1$ of $v$. This can be written as the following FOLB query:

\begin{align*}
  \poc & \equiv \forall  v\forall w\forall x\forall y \forall z \squ(w,x,y,z)  \\
       & \quad  \implies
         \left(\B(w,z,v) \wedge \B(w,x,v) \implies \B(z,y,v) \wedge \B(x,y,v)\right) \\
       & \quad \quad
         \wedge
         \left(\B(z,w,v) \wedge \B(x,w,v) \implies \B(y,z,v) \wedge \B(y,x,v))\right).
\end{align*}

Obviously, the $2$-interval condition can also be written as a
FOLB-query $\tict$ which is long but can be obtained in a
straightforward way from the definition.

By the result of~\cite{Ch_basis}, a graph is the basis graph of an even $\Delta$-matroid if and only if it satisfies
the Positioning Condition ($\PoC$) and the following two conditions:
\begin{itemize}
\item\emph{$2$-Interval Condition $4$} ($\TiCQ$): for any two vertices
  $u,v$ at distance $2$, $I(u,v)$ contains a square and is a subgraph of a
  $4$-octahedron.
\item\emph{Link Condition} ($\LiC$): the neighborhood of each vertex
  induces the line graph of a graph.
\end{itemize}

For the same reasons as for ($\TiCT$), ($\TiCQ$) can be described by a
FOLB-query.

By a theorem of Beineke~\cite{Bei}, a graph is a line graph if and only if does not contain a graph from a
list of nine forbidden induced subgraphs
$W_5=F_0, F_1, F_2, \ldots, F_8$ (see~\cite[Fig.1]{Ch_basis}). For each $F_i$, let $F_i'$ be the
graph obtained from $F_i$ by adding a universal vertex to $F_i$.  A
graph $G$ satisfies the link condition if and only if $G$ does not
contain any $F_i'$ as an induced subgraph. Consequently, basis graphs
of matroids and of even $\Delta$-matroids can be described by the following queries:
\begin{align*}
  \matroid & \equiv \tict \wedge \poc\\
  \deltamatroid & \equiv  \ticq \wedge \poc \wedge \neg {\subgraph}_{F'_0, F'_1\ldots, F'_8}.
\end{align*}

Notice also that Johnson graphs are the basis graphs of uniform matroids. In view of the characterizations of basis graphs
of matroids, Johnson graphs are exactly the graphs satisfying ($\PoC$) and in which each 2-interval is a 3-octahedron.
Analogously, the halved cubes are the graphs satisfying the Positionning Condition, in which each 2-interval is a 3-octahedron, and
in which the neighborhood of each vertex is the complete graph. Clearly, these versions of ($\TiCT$), ($\TiCQ$), and ($\LiC$)
are FOLB-definable.

\section{Partial cubes, their subclasses and superclasses}\label{sec:partial-cubes}
In this section, we present the FOLB-definability of partial cubes and their main subclasses and super-classes, which alltogether constitute another important part of Metric Graph Theory. Subclasses of partial cubes are the following classes of graphs:  median, ample, oriented matroids, complexes of oriented matroid, bipartite Pasch, bipartite Peano, cellular, hypercellular, two-dimensional partial cubes, partial cubes of VC-dimension $d$, almost median, and semimedian graphs.  Partial Hamming graphs is a superclass
of partial cubes containing quasi-median graphs as a subclass. Partial Johnson graphs, partial halved-cubes, and $\ell_1$-graphs are three very important and general super-classes of partial cubes. Basis graphs of matroids are partial Johnson graphs and basis graphs of even $\Delta$-matroids are partial halved-cubes.

\subsection{Partial cubes and partial Hamming graphs}
The \emph{partial cubes} are the graphs which admits an isometric embedding into hypercubes. Even if at first look it seems that this class of graphs is quite
restricted, many classes of graphs arising in geometry or combinatorics are partial cubes. Already Deligne~\cite{De} noticed that the number of hyperplanes separating
two regions defined by arrangements of hyperplanes in ${\mathbb R}^d$ coincide with the distance between these regions in the graph of regions, whence these graphs
are partial cubes. Partial cubes also occur in the paper of Graham and
Pollak~\cite{GrPo} about the compact routing in graphs.

Partial cubes have been nicely characterized by Djokovi\'{c}~\cite{Dj} via the convexity of half-spaces. Namely, he proved that a graph $G=(V,E)$ is a partial cube if and only if
$G$ is bipartite and for any edge $uv$ the sets $W(u,v)=\{ x\in V: d(x,u)<d(x,v)\}$ and $W(v,u)=\{ x\in V: d(x,v)<d(x,u)\}$ are complementary halfspaces of $G$.
Djokovi\'{c} also introduced the parallelism relation between the edges of a partial cube: two edges $e=uv$ and $e'=u'v'$ of a bipartite graph $G$ are in relation $\Theta$
if $W(u,v)=W(u',v')$ and $W(v,u)=W(v',u')$. The relation $\Theta$ is reflexive and symmetric and it is shown in~\cite{Dj} that $\Theta$ is transitive (i.e., is an equivalence relation)
if and only if $G$ is a partial cube.

Djokovi\'{c}'s characterization of  partial cubes in terms of  convexity of the sets $W(u,v)$ and $W(v,u)$ leads to the FOLB-definability of partial cubes:

\begin{align*}
  \convW(u,v)
  & \equiv \forall x \forall y \forall z \B(x,u,v) \wedge \B(y,u,v) \wedge \B(x,z,y) \implies \B(z,u,v) \\
  \convnonW(u,v)
  & \equiv  \forall x \forall y \forall z (\neg \B(x,u,v)) \wedge (\neg \B(y,u,v)) \wedge \B(x,z,y) \implies (\neg \B(z,u,v)).
\end{align*}
\begin{align*}
  \pcube
  & \equiv \bipartite \wedge  \forall u \forall v \E_{\B}(u,v) \implies \convW(u,v) \wedge \convnonW(u,v) \\
  & \equiv \bipartite \wedge  \forall u \forall v \E_{\B}(u,v) \implies \convW(u,v) \wedge \convW(v,u) \\
  & \equiv \bipartite \wedge  \forall u \forall v \E_{\B}(u,v) \implies \convW(u,v).
\end{align*}

Recall that \emph{Hamming graphs} are Cartesian products of cliques. Consequently, hypercubes are binary Hamming graphs, i.e., Cartesian products of edges.
The study of \emph{partial Hamming graphs} (i.e., of graphs which are isometrically embeddable into Hamming graphs) was initiated by Winkler~\cite{Winkler}.
Answering a question raised in~\cite{Winkler}, several characterizations of partial Hamming graphs have been given in the papers~\cite{Ch_hamming} and~\cite{Wil}.
Quasi-median graphs are partial Hamming graphs.

For FOLB-definability of partial Hamming graphs, we  will use the following  characterization from~\cite{Ch_hamming}, which is in the spirit of  Djokovi\'{c}'s theorem.
For any edge $uv$ of a graph $G$, let $W_=(u,v)=\{ x\in V: d(x,u)=d(x,v)\}$.  Then the sets $W(u,v),W(v,u),$ and $W_=(u,v)$ partition the vertex-set of $G$.
According to~\cite{Ch_hamming}, a graph $G$ is a partial Hamming graph if and only of for any edge $uv$ the sets $W(u,v),W(v,u),W_{=}(u,v)$ and their complements
are convex (i.e., are halfspaces of $G$).
\begin{align*}
  \convWeq(u,v)
  & \equiv \forall x \forall y \forall z \neg (\B(x,u,v) \vee \B(x,v,u))
    \wedge \neg (\B(y,u,v) \vee \B(y,v,u)) \wedge \B(x,z,y)\\
  & \implies
    \neg (\B(z,u,v) \vee \B(z,v,u))  \\
  \convnonWeq(u,v)
  & \equiv \forall x \forall y \forall z (\B(x,u,v) \vee \B(x,v,u)) \wedge
    (\B(y,u,v) \vee \B(y,v,u))\wedge \B(x,z,y) \\
  & \implies (\B(z,u,v) \vee \B(z,v,u)).
  \end{align*}
Consequently, we obtain:
\begin{align*}
  \phamming
  & \equiv \forall u \forall v \E_{\B}(u,v) \implies \convW(u,v) \wedge
    \convW(v,u) \wedge \\
  & \quad \convWeq(u,v) \wedge \convnonW(u,v) \wedge
    \convnonW(v,u) \wedge\\
  & \quad \convnonWeq(u,v) \\
  & \equiv \forall u \forall v \E_{\B}(u,v) \implies \convW(u,v) \wedge
    \convWeq(u,v) \\
  & \quad \wedge \convnonW(u,v) \wedge \convnonWeq(u,v).
\end{align*}

An \emph{$\ell_1$-graph} is a graph which admits an isometric embedding into some ${\mathbb R}^n$ endowed
with the $\ell_1$-metric. A \emph{partial Johnson graph} is a graph which admits an isometric embedding
into a Johnson graph and a \emph{partial halved cube} is a graph which admits an isometric embedding into a halved cube.
Notice that basis graphs of matroids are partial Johnson graphs and basis graphs of even Delta-matroids are
partial halved cubes~\cite{Ch_basis}. Weakly median graphs are $\ell_1$-graphs~\cite{BaCh_wmg}. For other classes
of $\ell_1$-graphs, see the survey~\cite{BaCh_survey}.  The $\ell_1$-graphs have been characterized by
Shpectorov~\cite{Shp}, who first proved that $\ell_1$-graphs are exactly
the graphs which admit a scale $k$ embedding into a hypercube for some finite $k$. Second, he proved that a
graph $G$ is an $\ell_1$-graph if and only if $G$ is isometrically embeddable into the Cartesian product of
complete graphs, hyperoctahedra, and halved cubes. The factors and the embedding can be efficiently constructed
from $G$, which lead to a polynomial time recognition of $\ell_1$-graphs. Nevertheless the characterization of
$\ell_1$-graphs given in~\cite{Shp} cannot be used for FOLB-definability
of $\ell_1$-graphs. The characterization of partial halved cubes is also open and was formulated as an open
question already in~\cite{DeLa}.  A Djokovi\'{c}-like characterization of partial Johnson graphs was given
in~\cite{Ch_Johnson}, where it was proved that a graph $G$ is a partial Johnson graph if and only if (1) for
any edge $uv$ the set $W_=(u,v)$ consists of at most two connected components $W'_=(u,v),W''_=(u,v)$ (which
are allowed to be empty) such that $(W(u,v)\cup W'_=(u,v),W(v,u)\cup W''_=(u,v))$ and $(W(u,v)\cup W''_=(u,v),W(v,u)\cup W'_=(u,v))$
are pairs of complementary halfspaces of $G$ and (2) a specially defined atom graph of $G$ is the line graph of a
bipartite graph. While condition (1) is FOLB-definable, condition (2) (which is in the same vein as the link condition
for basis graphs of matroids, which was fortunately redundant) is not. In fact, in the last section we will show that
being a partial Johnson graph is not FOLB-definable. This is mainly due to the fact that even wheels are partial Johnson
graphs and odd wheels are not (but even and odd wheels are embeddable into halved cubes). The case of partial halved
cubes and $\ell_1$-graphs is open:

\begin{question} Are classes of partial halved cubes and $\ell_1$-graphs FOLB-definable?
\end{question}

\subsection{Ample classes, OMs, and COMs}

That median graphs are partial cubes has been shown
in~\cite{Mu}. Another class of partial cubes extending median graphs
is the class of ample/lopsided/extremal graphs, introduced
independently in \cite{BaChDrKo,BoRa,La} (using different but
equivalent definitions). The basic definition of \emph{ample graphs}
is via shattering: the vertex-sets of such graphs are set families
$\mathcal F$ and such a graph $G$ is ample if whenever some set $Y$ is
obtained by shattering by $\mathcal F$, then $G$ contains a cube whose
coordinate set is $Y$. A closer to the subject of this section, is the
following characterization of ampleness provided in~\cite{BaChDrKo}: a
subgraph $G$ of a hypercube is ample if any two parallel cubes of $G$
can be connected in $G$ by a geodesic gallery. In case of vertices,
this is just the reformulation of being a partial cube. In general,
two \emph{parallel} cubes are two cubes of $G$ having the same set of
coordinates and a \emph{geodesic gallery} is a shortest path
consisting of parallel cubes only. Other combinatorial, recursive, and
metric characterizations of ample graphs can be found
in~\cite{BaChDrKo,BoRa,La}. We will use a nice characterization of
Lawrence~\cite{La} asserting that a subgraph $G$ of a hypercube is
ample if and only if any subgraph $H$ of $G$ that is closed by taking
antipodes is either a cube or empty.  Additionally to median graphs,
ample graphs comprise other interesting subclasses, for example the
convex geometries~\cite{EdJa}, an important combinatorial structure
arising in abstract convexity.

Oriented matroids (OMs), co-invented by Bland and Las Vergnas~\cite{BlLV}  and Folkman and
Lawrence~\cite{FoLa}, represent a unified combinatorial theory of orientations of ordinary matroids.
They capture the basic properties of sign vectors representing the circuits in a directed graph
and the regions in a central hyperplane arrangement in $ {\mathbb R}^m$. Oriented matroids are systems of
sign vectors (i.e., $\{ -1,0,+1\}$-vectors) satisfying three simple axioms (composition, strong elimination, and symmetry)
and may be defined in a multitude of ways, see the book by Bj\"orner et al.~\cite{BjLVStWhZi}.
Complexes of Oriented Matroids (COMs) were introduced not long ago
in~\cite{BaChKn} as a natural common generalization of ample classes and OMs. They satisfy the same set of three axioms as OMs,
where one of the axioms (of symmetry) was relaxed (to local symmetry).
Ample classes can be seen as COMs with cubical cells, while OMs are COMs with a single cell. In general COMs, the cells are OMs glued together in such a way that
the resulting cell complex is contractible. In the realizable setting, a COM corresponds to the intersection pattern of a
hyperplane arrangement (not necessarily central) with an open convex. Notice that  realizable ample classes
correspond to the intersection pattern of  a convex set with coordinate hyperplanes of ${\mathbb R}^m$~\cite{La} and that
realizable OMs correspond to the regions in a central hyperplane arrangement in $ {\mathbb R}^m$.  COMs contains other important
classes of geometric and combinatorial objects: it is shown in~\cite{BaChKn} that linear extensions of a poset or acyclic orientations of mixed
graphs, and CAT(0) Coxeter complexes (arising in geometric group theory~\cite{Da,HaPa}) are COMs.

The \emph{topes} of an OM or a COM $\mathcal L$ are the maximal sign vectors of $\mathcal L$ with respect to the order on sign
vectors where $0<-1$ and $0<+1$. In both cases, they are $\{-1,+1\}$-vectors, thus they can be viewed as subsets of vertices of the
hypercube of dimension $m$ (the size of the ground set). The \emph{tope graphs} of OMs and COMs are the subgraphs of the hypercube $Q_m$
induced by the set of topes. Generalizing the observation of~\cite{De}, it was shown in~\cite{BjEdZi} (see also~\cite{BjLVStWhZi}) that
tope graphs of OMs are partial cubes. In fact, tope graphs of OMs are centrally-symmetric and are isometric subgraphs of $Q_m$, i.e.,
are antipodal partial cubes~\cite{BjEdZi}. A characterization of tope graphs of OMs was given in~\cite{dS}.
That tope graphs of COMs are also partial cubes has been shown in~\cite{BaChKn}.

A nice characterization of tope graphs of COMs was recently given by
Knauer and Marc~\cite{KnMa}. This characterization is nice also
because, similarly to
how COMs unify ample sets and OMs, this  result unifies the characterization of ample graphs of Lawrence~\cite{La} and the characterization of tope graphs of
OMs of da Silva~\cite{dS}. An \emph{antipodal graph} is a graph such that for any vertex $v$ there exists a vertex $v^{\bot}$ such that any vertex $x$ of $G$ is metrically between
$v$ and $v^{\bot}$, i.e.. $x\in I(v,v^{\bot})$. Then the result of~\cite{KnMa} says that a partial cube $G$ is a tope graph of a COM if and only if any antipodal
subgraph of $G$ is gated. To obtain the characterization of ample graphs of~\cite{La}, additionally we have to require that each antipodal subgraph of $G$ is a cube. To obtain the
characterization of tope graphs of OMs of~\cite{dS} one has to require that $G$ itself is antipodal.

Since antipodal graphs are necessarily intervals,  first we write two
predicates $\antipodaluv(u,v)$ and $\gated(u,v)$, which are true if and only if the
subgraph induced by $I(u,v)$ is respectively antipodal or gated. Then we write the predicate $\cube(u,v)$ which is true if and only if
$I(u,v)$ induces a cube. This uses the characterization of hypercubes as median graphs in which each 2-interval is a square, i.e., as thick median graphs~\cite{BaMu_02}.

\begin{align*}
  \antipodaluv(u,v) &\equiv \forall x \B(u,x,v)
                    \implies \left(\exists y \forall z \B(u,z,v) \iff
                    \B(x,z,y)\right) \\
  \gated(u,v) & \equiv \forall x \exists x'
                \B(u,x',v) \wedge \left(\forall y \B(u,y,v) \implies
                \B(x,x',y)\right)\\
 \cube(u,v) &  \equiv \big( \forall x \forall y \forall z \B(u,x,v) \wedge \B(u,y,v)
              \wedge \B(u,z,v) \\
                    &\quad \implies \exists m \med(x,y,z,m) \wedge
               \B(u,m,v) \big)\\
                  & \quad \wedge \big( \forall x \forall y \B(u,x,v) \wedge
                    \B(u,y,v) \wedge {\dist}_2(x,y)\\
              & \quad \implies \exists a
                    \exists b \, a\neq b \wedge \E_{\B}(a,x) \wedge
                    \E_{\B}(a,y)
\wedge \E_{\B}(b,x) \wedge \E_{\B}(b,y) \big).
\end{align*}
This leads us to the  FOLB-definability of ample graphs, tope graphs of OMs, and tope graphs of COMs:
\begin{align*}
 \com & \equiv \pcube \wedge  \forall u \forall v \antipodaluv(u,v) \implies
         \gated(u,v)  \\
  \ample & \equiv \com \wedge \forall u \forall v \antipodaluv(u,v)
           \implies \cube(u,v)\\
  \antipodalglobal & \equiv \forall x \exists y \forall z \B(x,z,y)\\
  \om & \equiv \com \wedge \antipodalglobal.
\end{align*}

\subsection{Pasch, Peano, cellular and hypercellular graphs}
Some classes of graphs in Metric Graph Theory arise from the theory of abstract convexity and are motivated
by the properties of classical (Euclidean) convexity. In previous section, we already considered graphs with convex balls and bridged graphs, which are discrete analogs of the fundamental properties of Euclidean convexity that balls are convex and that neighborhoods of convex sets are also convex.
One fundamental property of Euclidean convexity is separability: any two disjoint convex sets can be separated by complementary halfspaces and thus
can be separated by a hyperplane. Its generalization to topological vector spaces is the Hahn–Banach theorem.  One consequence of this separability
result is that any convex set is the intersection of all half-spaces containing it. Another geometric property of Euclidean convexity is that
the convex hull of a convex set $A$ and a point $x$ is the union of all segments $[x,p]$ with $p\in A$.

A \emph{convexity structure}~\cite{So,VdV} is a pair $(X,{\mathcal C})$, where $X$ is a set and $\mathcal C$ is a family of subsets of $X$ containing
$\varnothing$ and $X$, and closed by taking intersections and directed unions. The members of $\mathcal C$ are called \emph{convex subsets}.
Given $S\subset X$, the intersection of all convex subsets containing $S$ is called the \emph{convex hull} of $S$ and is denoted by $\conv(S)$.
A convexity $\mathcal C$ is called \emph{domain finite} if a set is convex if and only if
it contains the convex hull of its finite subsets and of
\emph{arity k} if a set is convex if and only if it contains the convex hull of its subsets of size at most $k$. Geodesic convexity in
graphs and, more generally, the convexity in interval spaces,  are examples of  convexity structures of arity 2.

A convexity structure is \emph{join-hull commutative (JHC)} if for any
convex set $A$ and any point $x$, $\conv(A\cup\{x\})=\bigcup_{a\in A}
\conv(a,x)$. In the case of interval spaces and graphs one can consider a
stronger version: convexity structure is \emph{join-hull commutative} if for any convex set $A$ and any point $x$, $\conv(A\cup\{x\})=\bigcup_{a\in A} I(a,x)$. The definition implies that join-hull
commutative domain finite convex structures have arity 2. Join-hull commutativity is also equivalent to the more general property
that $\conv(A\cup B)=\bigcup_{a\in A,b\in B} \conv(a,b)$~\cite{KaWo}. It was shown in~\cite{Cal}  that join-hull commutativity is
equivalent to the classical axiom in geometry, called the \emph{Peano axiom}: \emph{for any $u,v,w\in X$, for any $x\in \conv(w,v)$,
and for any $y\in \conv(u,x)$ there exists $z\in \conv(u,v)$ such that $y\in \conv(w,z)$.} If the convexity structure is defined by an interval structure, then the Peano axiom can be rewritten as:
\emph{for any $u,v,w\in X$, for any $x\in I(w,v)$, and for any $y\in I(u,x)$ there exists $z\in I(u,v)$ such that $y\in I(w,z)$.}

For convexity structures there exist several non-equivalent separability properties~\cite{So,VdV}:
\begin{itemize}
\item{$(S_2)$} any distinct two points $x,y$ can be separated by complementary halfspaces, i.e., there exist $H,X\setminus H\in {\mathcal C}$ such that $x\in H$ and $y\in X\setminus H$;
\item{$(S_3)$} any convex set $C$ and point $x\notin C$ can be separated by complementary halfspaces;
\item{$(S_4)$} any disjoint convex sets $A,B\in {\mathcal C}$ can be separated by complementary halfspaces.
\end{itemize}
$(S_4)$ is often called \emph{Kakutani separation property}~\cite{VdVS4,VdV}.
One can easily see that $(S_3)$ is equivalent to the fact that each convex set $C$ is the intersection of the halfspaces containing $C$. Ellis~\cite{Ell} proved that
if a convexity is join-hull commutativity, $(S_4)$ is equivalent to the Pasch axiom from geometry (this result was rediscovered by~\cite{VdVS4}). However not any $(S_4)$ convexity is JHC:
for example, the geodesic convexity of the graph consisting of a 4-wheel $W_4$ plus a vertex $x$ adjacent to two adjacent vertices of the 4-cycle
$C_4$ of $W_4$ is $(S_4)$  but not JHC: if the convex set $A$
consists of the edge of $C_4$ not incident to $x$, then $conv(A\cup \{ x\})$ consists of all vertices of the graphs, while $\bigcup_{a\in A} I(a,x)$
does not contain the central vertex of $W_4$.
However, it was shown in~\cite{Ch_dfc,Ch_separation} that for all convexity structures of arity 2 (in particular for geodesic convexity in graphs), the Pasch axiom is equivalent to $(S_4)$
(analogous conditions were also established for all arities).  In case of geodesic convexity in graphs, $(S_4)$ implies that the intervals are convex,
thus the Pasch axiom can be written in terms of intervals. Recall that the Pasch axiom is one of the main axioms in Tarski's and Hilbert's systems of axioms of geometry.
For convexities of arity 2, the \emph{Pasch axiom} can be formulated as follows:  \emph{for any $u,v,w\in X$, for any $x\in \conv(w,u), y\in \conv(w,v)$, there
exists $z\in \conv(u,y)\cap \conv(v,x)$.}  For graphs  this is equivalent to: \emph{for any $u,v,w\in X$, for any $x\in I(w,u), y\in I(w,v)$, there exists $z\in I(u,y)\cap I(v,x)$.}

The weakly median graphs are exactly the Pasch and also Pasch-Peano weakly modular graphs~\cite{Ch_separation}. Bipartite Pasch graphs are partial cubes and they have been
characterized in~\cite{Ch_bipartite,Ch_separation} in terms of pc-minors. It was shown in~\cite{KnMa} that bipartite Pasch graphs are tope graphs of COMs.
Median graphs are Pasch. Bipartite Peano graphs are also partial cubes~\cite{Ch_thesis,Po} and they have been investigated
in~\cite{Po}. Pasch-Peano graphs have been investigated in~\cite{BaChVdV} (see these results in the book~\cite{VdV}).

One can easily show (see~\cite{AlKn,BaS3,Ch_bipartite,Ch_thesis}) that for bipartite graphs the conditions $(S_2)$ and $(S_3)$ are equivalent
and they are both equivalent to the fact that $G$ is a partial cube. However, beyond bipartite graphs, no characterization of graphs or convexity structures
satisfying $(S_3)$ (similar to that for $(S_4)$ provided in~\cite{Ch_dfc,Ch_separation}) is known. On the other hand,  the importance of $(S_3)$ was recently highlighted
by a general result of~\cite{MoYe} showing that in  $(S_3)$ convexity structures, Radon number characterizes the existence of weak $\epsilon$-nets.
Nevertheless, under join-hull commutativity it was shown in~\cite{Ch_dfc} that $(S_3)$ is equivalent to the following Pasch-like condition, called \emph{sand-glass axiom} in~\cite{VdV}:
for any $u,v,u',v',y$ such that $y\in \conv(u,u')\cap \conv(v,v')$ and for any $x\in \conv(u,v)$ there exists $x'\in \conv(u',v')$ such that $y\in \conv(x,x')$.

Finally, we will say that a graph $G$ is a graph with \emph{convex intervals} if all intervals $I(u,v)$ of $G$ are convex sets. Graphs satisfying  separability properties $(S_3)$ and
$(S_4)$ as well as graphs satisfying (JHC) with intervals are graphs with convex intervals. Property of having convex intervals is FOLB-definable:
\begin{align*}
  \convint & \equiv \forall u\forall v\forall x\forall y \forall z \left( \B(u,x,v)\wedge \B(u,y,v)\wedge \B(x,z,y)\right)\implies \B(u,z,v).
\end{align*}

 The join hull commutativity is also FOLB-definable:
 \begin{align*}
  \jhc & \equiv \peano \equiv \forall u\forall v\forall w\forall x \forall y \left( \B(v,x,w)\wedge \B(u,y,x)\right) \implies \exists z \left(\B(u,z,v)\wedge \B(w,y,z)\right).\\
\jhc & \bipartite \equiv \peano \wedge \pcube.
\end{align*}

Finally, summarizing the results about the separability properties $(S_2),(S_3)$ and $(S_4)$, we obtain the following FOLB-definability:
\begin{align*}
\pasch & \equiv \forall u \forall v \forall w \forall x \forall y \left( \B(u,x,w)\wedge \B(v,y,w)\right)\implies \exists z \left( \B(u,z,y)\wedge \B(v,z,x)\right)\\
\sandglass & \equiv \forall u \forall v \forall u' \forall v' \forall x\forall y \left(\B(u,y,u')\wedge \B(v,y,v')\wedge \B(u,x,v)\right)\\
& \quad \implies \exists x'\left( \B(u',x',v')\wedge \B(x,y,x')\right)\\
\sepcc & \equiv \pasch\\
\sepvc\wedge\peano &\equiv \sandglass\wedge \peano\\
\paschpeano & \equiv \pasch \wedge \peano\\
\bpasch & \equiv \pasch \wedge \pcube\\
\bpeano & \equiv \peano \wedge \pcube\\
\pcube & \equiv \sepvv\wedge \bipartite \equiv \sepvc \wedge \bipartite \\
\wmedian & \equiv \wmpasch \equiv \paschpeano \wedge \wmodular.
\end{align*}

\begin{question} Are the separation properties $(S_2)$ and $(S_3)$ on graphs FOLB-definable? Can $(S_3)$ be characterized by a condition on a fixed number of vertices?
\end{question}

\emph{Cellular graphs} are the bipartite graphs whose distance is totally decomposable in the sense of Bandelt and Dress~\cite{BaDr}. Their structure
has been investigated in~\cite{BaCh_cellular}, where it is shown that in such graphs all isometric cycles are gated and that the graph itself can be obtained from its
isometric cycles via gated amalgamations. Moreover, cellular graphs are characterized in~\cite{BaCh_cellular} in terms of convexity:
those are exactly the bipartite graphs for which $\conv(S)=\bigcup_{x,y\in S}I(x,y)$ holds for any set $S$ and equivalently are exactly
the bipartite graphs for which $\conv(u,v,w)=I(u,v)\cup I(v,w)\cup I(w,u)$ for any three vertices $u,v,w$. This last characterization shows FOLB-definability of cellular graphs:
\begin{align*}
\cellular & \equiv \bipartite \wedge \forall u \forall v \forall w \forall x \forall y \forall z \big(\B(u,x,v) \vee \B(v,x,w) \vee \B(w,x,u)\big)\wedge \B(x,z,y)\\
& \quad \wedge \big( \B(u,y,v)\vee \B(v,y,w)\vee \B(w,y,u)\big)\implies \big( \B(u,z,v)\vee \B(v,z,w)\vee \B(w,z,u)\big).
\end{align*}

\subsection{Almost-median, semi-median graphs and netlike partial cubes}

Almost-median graphs and semi-median graphs  introduced in~\cite{imrich98} are two classes of partial cubes generalizing  median graphs and ample partial cubes.  Given an edge $uv$ of a partial cube $G$, the \emph{boundary} $U(u,v)$
of the halfspace $W(u,v)$ is the set of vertices of $W(u,v)$ having a
neighbor in the complementary halfspace $W(v,u)$, i.e., the vertices
of $W(u,v)$ that are incident to an edge in the $\Theta$-class of the
edge $uv$. The boundaries $U(u,v)$ and $U(v,u)$ induce isomorphic
subgraphs of $G$ and these subgraphs are isomorphic to the hyperplane
$H(uv)$ of the $\Theta$-class of the edge $uv$. The \emph{hyperplane}
$H(uv)$ has the middles of the edges of the $\Theta$-class
$\Theta(uv)$ and the middles of the edges $u'v'$ and $u''v''$ are
adjacent if $u'v'v''u''$ is a square of $G$. The hyperplanes play an
important role in the theory of median graphs viewed as CAT(0) cube
complexes. Median graphs are in fact the partial cubes in which all
boundaries are gated/convex~\cite{Ba_man}. The boundaries and the
hyperplanes of ample partial cubes are also ample and this
characterizes ample partial cubes~\cite{BaChDrKo}. Almost-median and
semi-median graphs are a generalization of median and ample classes. A
partial cube $G$ is \emph{almost-median} (respectively,
\emph{semi-median}) if all boundaries are isometric (respectively,
connected) subgraphs of $G$, i.e., they are partial cubes.  From their
definition, almost-median graphs can be described by the following
FOLB-query:
\begin{align*}
  \amedian &\equiv \pcube \wedge \forall u  \forall v \forall x' \forall
             y'\forall x'' \forall   y''
             \E_{\B}(u,v) \wedge \B(x',u,v) \wedge  \B(x'',v,u) \wedge
             \E_{\B}(x',x'') \\
           & \quad \wedge \B(y',u,v)  \wedge \B(y'',v,u)  \wedge \E_{\B}(y',y'') \wedge \neg \E_{\B}(x',y')
             \implies \exists z'\exists z'' \B(x',z',y')\\
           & \quad \wedge \B(x'',z'',y'') \wedge
             \E_{\B}(z',z'')
\end{align*}

Almost-median graphs have characterized in~\cite{Br,KlSh,Pe}.  By a
result of~\cite{KlSh} and~\cite{Pe}, a partial cube is almost-median
if and only if it satisfies the following almost quadrangle condition.
\begin{itemize}
\item\emph{Almost Quadrangle Condition} ($\AQC$): for any four
  vertices $v,x,y,u$ such that $d(v,x)=d(v,y)=d(v,u)-1$ and
  $u\sim x,y$, $x\nsim y$, there exists $z\in I(x,v)$ and
  $w \in I(u,z)$ such that $xzwu$ is a square of $G$;
\end{itemize}

The corresponding FOLB query is:
\begin{align*}
  \aqc   \equiv &  \forall v \forall x \forall y \forall u \E_{\B}(u,x)
                  \wedge \E_{\B}(u,y) \wedge \B(u,x,v) \wedge \B(x,u,y)
                  \wedge \B(u,y,v) \wedge\\
                &x \neq y \wedge (\neg \E_{\B}(x,y)) \implies \exists z \exists w  \B(x,z,v)
                  \wedge  \squ(u,x,z,w)
\end{align*}
Consequently almost-median graphs can also be described by the
following FOLB-query.
\begin{align*}
  \amedian &\equiv  \pcube \wedge \aqc\\
\end{align*}

\begin{question}
  Are semi-median graphs FOLB-definable?
\end{question}

Netlike partial cubes form another class of partial cubes that is defined by some property of its boundaries.

For a graph $H$, we denote by $CV(H)$ the set of vertices belonging to
a cycle of $H$. We denote also $3V(H)$ the set of vertices of degree
at least $3$ in $H$.  Given a subset $A$ of the vertices of a graph
$G$, we denote by $I(A)$ the subgraph induced by the set
$\bigcup_{x,y \in A}I(x,y)$. Notice that the convex hull $\conv(A)$ of
$A$ can be obtained by iterating applications of the operator $I$ and
that $I(\conv(A)) = \conv(A)$. Finally, a set $A$ of vertices of $G$
is \emph{ph-stable}~\cite{Po1} if for all $u,v \in I(A)$, there exists
$w \in A$ such that $v \in I(u,w)$.  A set $A$ of vertices is
\emph{$\C$-convex} if $CV(I(A)) \subseteq A$. Analogously, a set $A$
of vertices is \emph{degree-$3$-convex} if $3V(I(A)) \subseteq A$.
According to Polat~\cite{Po1}, a partial cube $G$ is \emph{netlike} if
for every edge $uv$, the boundaries $U(u,v)$ and $U(v,u)$ are $\C$-convex.
Median graphs, cellular bipartite graphs and benzenoids are example of
netlike partial cubes. By~\cite[Theorem 3.8]{Po1}, a partial cube $G$
is netlike if and only if for each edge $uv$ of $G$, the boundaries
$U(u,v)$ and $U(v,u)$ are ph-stable and degree-$3$-convex. This
characterization allows to show that netliked partial cubes are
FOLB-definable.

\begin{align*}
  \boundary(u,v,z) & \equiv \exists z' \B(v,u,z) \wedge \B(u,v,z') \wedge \E_{\B}(z,z')\\
  \intBoundary(u,v,z) & \equiv \exists z' \exists z'' \boundary(u,v,z') \wedge \boundary(u,v,z'')\wedge \B(z',z,z'')\\
  \phstable(u,v) & \equiv \forall x \forall y \intBoundary(u,v,x) \wedge \intBoundary(u,v,y)\implies \exists w \boundary(u,v,x) \wedge \B(x,y,w)\\
  \degTconv(u,v) & \equiv \forall z \forall y_1 \forall y_2 \forall y_3 \intBoundary(u,v,z) \wedge \intBoundary(u,v,y_1) \wedge \intBoundary(u,v,y_2)\\
  & \quad \wedge \intBoundary(u,v,y_3)\wedge \E_{\B}(z,y_1) \wedge \E_{\B}(z,y_2) \wedge \E_{\B}(z,y_3) \wedge y_1 \neq y_2\\
  & \quad  \wedge y_1 \neq y_3 \wedge y_2 \neq y_3\implies \boundary(u,v,z) \\
  \netpcube & \equiv \pcube \wedge \forall u \forall v \left(\E_{\B}(u,v) \implies \phstable(u,v) \wedge \degTconv(u,v)\right)
\end{align*}

For other results on netlike partial cubes, see the series of papers of Polat~\cite{Po1,Po2,Po3,Po4}.

\section{Gromov hyperbolic graphs and their subclasses}\label{sec:hyperbolicity}

\subsection{Hyperbolic graphs} Hyperbolic metric spaces and Hyperbolic
graphs have been defined by Gromov~\cite{Gr} and is a fundamental
object of study in geometric group theory.  A graph $G$ is
$\delta$--\emph{hyperbolic}~\cite{BrHa,Gr} if for any four vertices
$u,v,x,y$ of $X$, the two larger of the three distance sums
$d(u,v)+d(x,y)$, $d(u,x)+d(v,y)$, $d(u,y)+d(v,x)$ differ by at most
$2\delta \geq 0$. The hyperbolicity of a graph $G$ is
$\delta^*(G) = \inf \{\delta: \text{$G$ is $\delta$-hyperbolic}\}$.
For graphs (and geodesic metric spaces), $\delta$--hyperbolicity can
be defined (up to a constant factor) as spaces in which all geodesic
triangles are $\delta$--slim.  In a graph $G$ a \emph{geodesic
  triangle} $\delta(x,y,z)$ is a triplet of vertices $x,y,z$ and a
triplet of geodesics (shortest paths) $[x,y]$, $[x,z]$, and $[y,z]$.
A geodesic triangle $\Delta(x,y,z)$ is called $\delta$--{\it slim} if
for any point $u$ on the geodesic $[x,y]$ the distance from $u$ to
$[x,z]\cup [z,y]$ is at most $\delta$. If a graph $G$ has
$\delta$-slim triangles, then $G$ is $2\delta$-hyperbolic and
conversely if $G$ is $\delta$-hyperbolic, then $G$ has $3\delta$-slim
triangles (see~\cite{BrHa,GhHa}).  There are many other
characterizations of $\delta$-hyperbolicity~\cite{BrHa,Gr,GhHa} such
as characterizations via $\delta$-thin triangles, thinness of
intervals, or linear isoperimetric inequality.

We give a new ``definition'' of hyperbolicity that relaxes the
slimness of geodesic triangles.  We say that a graph $G$ is
\emph{interval-$\delta$-slim} if for any triplets of vertices $x,y,z$
and for each $u \in I(y,z)$, there exists $v \in I(x,y) \cup I(x,z)$
such that $d(u,v) \leq \delta$.

\begin{lemma}\label{lem-int-delta}
  If $G$ is interval-$\delta$-slim, then $G$ has $3 \delta$-slim
  triangles and $G$ is $6\delta$-hyperbolic.

  Conversely, if $G$ is $\delta$-hyperbolic, then $G$ has
  $3\delta$-slim triangles and $G$ is interval-$3\delta$-slim.
\end{lemma}

\begin{proof}
  Suppose that $G$ is interval-$\delta$-slim. Observe first that $G$
  has $2\delta$-thin intervals, i.e., for any $u, v \in I(x,y)$ such
  that $d(u,x) = d(v,x)$, we have $d(u,v) \leq 2\delta$. Indeed
  consider the triplet $x,y,v$ and the vertex $u \in I(x,y)$. Since
  $G$ is interval-$\delta$-slim, there exists a vertex
  $t \in I(v,x) \cup I(v,y)$ with $d(u,t)\leq \delta$. Without loss of
  generality, suppose that $t \in I(v,x)$. Consequently,
  $d(v,t)+d(t,x) = d(v,x) = d(u,x) \leq d(u,t) + d(t,x)$ and thus
  $d(v,t) \leq d(u,t) \leq \delta$.  Now pick any geodesic triangle
  $\Delta(x,y,z)$ and $u \in [y,z]$. Since $G$ is
  interval-$\delta$-slim, there exists $v \in I(x,y)\cup I(x,z)$ with
  $d(u,v) \leq \delta$. Let $v \in I(x,y)$ and let $w \in [x,y]$ such
  that $d(x,v) = d(x,w)$. By the $2\delta$-thinness of the intervals,
  $d(v,w) \leq 2\delta$ and thus $d(u,w) \leq 3\delta$ by the triangle
  inequality. Consequently, $G$ has $3\delta$-slim triangles and is
  $6\delta$-hyperbolic since graphs with $\delta$-slim triangles are
  $2\delta$-hyperbolic.

  Conversely, if $G$ has $\delta$-slim triangles, then trivially $G$
  is interval-$\delta$-slim and the second assertion of the theorem
  follows from the fact that $\delta$-hyperbolic graphs have
  $3\delta$-slim triangles.
\end{proof}

It seems challenging to find a FOLB-query characterizing the graphs whose hyperbolicity is at most $\delta$.
However, we can characterize in FOLB
interval-$\delta$-slim graphs by the following query.

\begin{align*}
  {\intDeltaSlim}_{\delta} & \equiv  \forall x \forall y \forall z \forall u
                         \B(y,u,z) \implies \exists v (\B(x,v,y) \vee
                         \B(x,v,z)) \wedge {\dist}_{\leq \delta}(u,v)  \\
\end{align*}

By Lemma~\ref{lem-int-delta}, for any graph $G$, if $G$ satisfies
${\intDeltaSlim}_{\delta}$ then $\delta^*(G) \leq 6\delta$ and if $G$
does not satisfy ${\intDeltaSlim}_{\delta}$, then
$\delta^*(G) > \delta/3$. When the hyperbolicity $\delta^*(G)$ of $G$
is between $\delta/3$ and $6\delta$, $G$ can satisfy the query
${\intDeltaSlim}_{\delta}$ or not.

\subsection{Subclasses of hyperbolic graphs} In geometric group theory several types of hyperbolic graphs and
complexes are investigated. This is due to the fact that the groups acting geometrically
on a hyperbolic graph satisfy strong properties~\cite{BrHa,GhHa,Gr}. Among such classes of graphs are the curve graphs and the arc graphs.
The \emph{curve graph} of a compact oriented surface $S$ is the graph whose vertex set is the
set of homotopy classes of essential simple closed curves and whose edges correspond to
disjoint curves. The \emph{arc graph} of $S$  is the subgraph of the curve graph induced
by the vertices that are homotopy classes of arcs. It was proved in~\cite{MaMi} and~\cite{MaSc} that
curve graphs and arc graphs are hyperbolic, i.e., they have finite hyperbolicity (notice that these graphs are not
finite, they are not even locally finite). Later work by various authors gave alternate proofs of this fact and better
information on the hyperbolicity of the curve complex~\cite{Au,
  Bow}. Finally, it was shown in~\cite{HePrWe} that arc graphs
are 7-hyperbolic and that  curve graphs are 17-hyperbolic. Unfortunately, there is no characterization of arc graphs and
curve graphs, therefore their FOLB-definability is an open (and probably difficult) question. Hyperbolic systolic and CAT(0) cubical complexes
have been also investigated in geometric group theory and metric graph theory. For example, it was shown in~\cite{JaSw} that  7-systolic
complexes (i.e., systolic complexes not containing 6-wheels) are 11-hyperbolic. Later, in~\cite{ChDrEsHaVa} it was shown that in fact
they are 1-hyperbolic. It was shown in~\cite{ChDrEsHaVa} and~\cite{Ha} that a median graph $G$ (a 1-skeleton of a CAT(0) cube complex) is
$\delta$-hyperbolic if and only if $G$ does not contain isometrically embedded square  $\delta\times \delta$-grids. This kind of characterization
of hyperbolicity in terms of grids (triangular or square)  was generalized in~\cite{CCHO,ChDrEsHaVa} to all weakly modular graphs and
a sharp characterization of $\delta$-hyperbolic Helly graphs was obtained in~\cite{DrGa}.

Notice that trees are $0$-hyperbolic. Moreover, the $0$-hyperbolic
graphs are exactly the \emph{block-graphs}, i.e., the graphs in which
all blocks (2-connected components) are complete graphs. The next
hyperbolicity constant is $\frac{1}{2}$. The graphs whose
hyperbolicity is at most $\frac{1}{2}$ have been characterized in
\cite{BaCh_hyp} (where they are called $1$-hyperbolic graphs): these
are exactly the graphs with convex balls not containing 6 isometric
subgraphs $H_a, H_b, H_c, H_d, H_e, H_f$ (see Fig. 2 of
\cite{BaCh_hyp}), which we denote by $H_1,\ldots,H_6$.  The
1-hyperbolic graphs have not yet been characterized. This class
contains chordal graphs, 7-systolic graphs, and distance hereditary
graphs. A graph $G$ is \emph{distance hereditary}~\cite{BaMu_dh,Ho_dh}
if any induced subgraph of $G$ is an isometric subgraph, i.e., if any
induced path is a shortest path. Distance hereditary graphs have been
characterized in various ways in~\cite{BaMu_dh,Ho_dh}. For example, it
was shown in~\cite{BaMu_dh} that a graph $G$ is distance hereditary if
and only if $G$ does not contain the following graphs as induced
subgraphs: the cycles $C_k, k\ge 5$, the 3-fan, the house, and the
domino (see Fig. 1-3 of~\cite{BaMu_dh}).  Also it was shown in
\cite{BaMu_dh} that distance hereditary graphs are exactly the graphs
in which for any four vertices $u,v,x,y$ among the three distance sums $d(u,v)+d(x,y)$, $d(u,x)+d(v,y)$, $d(u,y)+d(v,x)$
at least two are equal and if the two smallest distance sums are
equal, then the largest differ from them by at most 2. This shows that
distance hereditary graphs are 1-hyperbolic. For FOLB-definability of
distance hereditary graphs we will use the following characterization
from~\cite{BaMu_dh}: for any three vertices $u,v,w$ at least two of
the following inclusions hold:
$I(u,v)\subseteq I(u,w)\cup I(w,v), I(u,w)\subseteq I(u,v)\cup
I(v,w)$, and $I(v,w)\subseteq I(v,u)\cup I(u,w)$. A subclass of
distance hereditary graphs is constituted by \emph{ptolemaic graphs}
\cite{Ho_pt,KaCh}.  Those are the graphs which verify the Ptolemaic
inequality from Euclidean geometry: for any four vertices $u,v,x,y$
$d(u,v)\cdot d(x,y)\le d(u,x)\cdot d(v,y)+d(u,y)\cdot d(v,x)$. It was
shown in~\cite{BaMu_dh} that the ptolemaic graphs are exactly the
chordal distance hereditary graphs and thus are exactly the chordal
graphs not containing the 3-fan $3F$, and they are distance hereditary
graphs without $C_4$ and $3F$.  Notice that ptolemaic graphs are exactly
the graphs whose geodesic convexity satisfies the Krein-Milman property
that any convex set is the convex hull of its extremal vertices~\cite{So_KM}.
Notice also  that block graphs are exactly
the ptolemaic graphs not containing an induced $K_4^-$.

We conclude with the definition of graphs with $\alpha_i$-metrics,
which been introduced and studied in~\cite{Ch_tr}. A graph $G$ has an
\emph{$\alpha_i$-metric}~\cite{Ch_tr} if for any edge $vw$ of $G$ and
any two vertices $u,x$ such that $v\in I(u,w)$ and $w\in I(v,x)$, the
inequality $d(u,x)\ge d(u,v)+d(w,x)+d(w,v)-i$ holds.  The motivation
again comes from Euclidean geometry: if $v,w$ are two close points and
we extend the segment $[v,w]$ to a segment $[u,w]$ via $v$ and to a
segment $[v,x]$ via $w$, then the points $v,w$ belong to the segment
$[u,x]$. Or more informally, if we shoot a ray from $w$ trough $v$ and
a ray from $v$ through $w$, then these two rays will define a
line. Therefore, the $\alpha_i$-metric shows how $d(u,x)$ is close to
$d(u,v)+d(v,w)+d(w,x)$. It was shown in~\cite{Ch_tr} that ptolemaic
graphs are exactly the graphs with $\alpha_0$-metrics and that chordal
graphs are graphs with $\alpha_1$-metric. The graphs with
$\alpha_1$-metrics have been characterized in~\cite{YuCh}: these are
exactly the graphs with convex balls not containing the graph $H_c$ as
an isometric subgraph (where $H_c$ is a graph from the list of
forbidden subgraphs for $\frac{1}{2}$-hyperbolicity of
\cite{BaCh_hyp}). It will be interesting to investigate in more
details the structure and the characterizations of graphs with
$\alpha_i$-metrics as has been done for hyperbolic graphs. Since
Euclidean spaces have $\alpha_0$-metrics it is clear that
$\alpha_0$-metrics are not hyperbolic. For graphs, the links between
$\delta$-hyperbolic graphs and graphs with $\alpha_i$-metrics are less
clear.

Now, we have collected all necessary material to present the FOLB-definability of all previously defined classes of graphs. The interval functions of block graphs,
ptolemaic graphs, and distance hereditary graphs have been provided in~\cite{block,ptolemaic,DHG}; our characterizations of those classes are different.

\begin{align*}
  \block & \equiv {\hyperbolic}_{0 } \equiv {\intDeltaSlim}_{0,0} \equiv \pt \wedge \neg {\subgraph}_{K_4^-} \\
  {\hyperbolic}_{\frac{1}{2}} & \equiv \cballs \wedge \neg {\isometric}_{H_a,H_b,H_c,H_d,H_e,H_f} \\
  {\alpha_0-\metric} & \equiv \pt \\
  {\alpha_1-\metric} & \equiv \cballs \wedge \neg {\isometric}_{H_c} \\
  \intIncl(u,v,w) & \equiv \forall x \B(u,x,v) \implies \B(u,x,w) \wedge \B(v,x,w) \\
  \dhg & \equiv \forall u \forall v \forall w
         (\intIncl(u,v,w) \wedge   \intIncl(u,w,v))\\
&         \vee (\intIncl(u,v,w) \wedge   \intIncl(v,w,u))
         \vee (\intIncl(u,w,v) \wedge   \intIncl(v,w,u))\\
  \pt & \equiv \dhg \wedge \neg {\subgraph}_{C_4,3F}\\
\end{align*}

Notice that the condition $\alpha_0$ is easily FOLB-definable: for any adjacent vertices $v,w$ of $G$ and
any two vertices $u,x$ such that $v\in I(u,w)$ and $w\in I(v,x)$ we have $v,w\in I(u,x)$. However, we do not know if
the class of graphs with $\alpha_i$-metrics with $i\ge 2$ is FOLB-definable.

Summarizing the results of the last three sections, we obtain the following result about classes of graphs from Metric Graph Theory, which are FOLB-definable:

\begin{theorem}\label{theorem1} The following classes of graphs are FOLB-definable:

\begin{itemize}
\item[(a)] weakly modular, modular, quasi-modular, pseudo-modular, strongly modular, sweakly modular;
\item[(b)] bipartite, median, pseudo-median, quasi-median, weakly median;
\item[(c)] bridged, weakly bridged, bucolic, graphs with convex balls, graphs with $\alpha_1$-metric;
\item[(d)] Helly, clique Helly, dual polar;
\item[e)] meshed, basis graphs of matroids, basis graphs of $\Delta$-matroids;
\item[(f)] partial cubes, partial Hamming graphs, ample, tope graphs of OMs and COMs;
\item[(g)] Pasch, graphs with $S_4$ separability, Peano, join-hull commutative, graphs with convex intervals,
graphs with sand-glass property, bipartite graphs with $S_2$, bipartite graphs with $S_3$;
\item[(h)] cellular, hypercellular, almost median, netlike partial cubes;
\item[(i)] graphs with interval $\delta$-slim triangles,  0-hyperbolic and $\frac{1}{2}$-hyperbolic graphs,
distance hereditary, ptolemaic, block graphs.
  \end{itemize}
\end{theorem}

In the next section, we will prove that some classes of graphs defined in previous sections are not FOLB-definable. On the other hand, we do not know if the following classes of graphs
are FOLB-definable: \emph{partial halved cubes, $\ell_1$-graphs, $S_2$ graphs, $S_3$ graphs, semi-median graphs, and graphs of hyperbolicity at most $\delta$.}

\section{Graph classes not expressable in FOLB}\label{notfolb}

In the previous sections we have seen that several queries as
$\acyclicity$ and $\bipartite$ are not FOL-definable but are
FOLB-definable. In this section, we prove that several queries remain
not definable in the extended logic FOLB. Recall that a graph $G$ is
\emph{chordal} if $G$ does not contain induced cycles of length $>3$.
A graph is \emph{Eulerian} if it has a circuit traversing each edge
exactly once; recall that a graph $G$ is Eulerian if and only if $G$
is connected and all vertices of $G$ have even degree. A graph $G$ is
\emph{planar} if $G$ admits a drawing in the plane such that two
non-incident edges do not intersect and two incident edges intersect
only in their common end. By Kuratowski's theorem, a graph $G$ is
planar if an only of $G$ does not have $K_{3,3}$ and $K_5$ as a minor.
A graph $G$ is \emph{dismantlable} if the vertices of $G$ can be
linearly ordered $v_1,\ldots,v_n$ such that for each $1<i\le n$ the
vertex $v_i$ is dominated by a vertex $v_j, j<i$ in the subgraph $G_i$
induced by $\{ v_1,\ldots,v_i\}$; $v_i$ is \emph{dominated} by $v_j$
if $v_j$ is a neighbor of $v_i$ and all other neighbors of $v_i$ in
$G_i$ are adjacent to $v_j$ (the order $v_1,\ldots,v_n$ is called a
\emph{dismantling ordering}). A graph $G$ admits a \emph{distance
  preserving ordering} (DPO for short) if its vertices can be linearly
ordered $v_1,\ldots,v_n$ such that for each $1<i\le n$ the subgraph
$G_i$ induced by $\{ v_1,\ldots,v_i\}$ is an isometric subgraph of
$G$.  Any dismantling ordering is a distance preserving ordering but
the converse is not true.  Finally, as defined above, recall that a
graph $G$ is a \emph{partial Johnson graph} if $G$ isometrically
embeds into a Johnson graph.

In this section, we consider the following queries:
\begin{itemize}
\item The $\chordal$ query is the Boolean query such that $\chordal(\bG)=1$ iff $\bG$ is a chordal graph;
\item The $\planar$ query is the Boolean query such that $\planar(\bG)=1$ iff $\bG$ is a planar graph;
\item The $\eulerian$ query  is the Boolean query such that $\eulerian(\bG)=1$ iff $\bG$ is an Eulerian graph;
\item The  $\dismantlable$ query is the Boolean query such that $\dismantlable(\bG)=1$ iff $G$ is a dismantlable graph;
\item The $\dpo$ query is the Boolean query such that $\dpo(\bG)=1$ iff $G$ admits a distance preserving ordering;
\item The $\partialJ$ query  is the Boolean query such that $\partialJ(\bG)=1$ iff $\bG$ is a partial Johnson graph.
\end{itemize}

Planar and Eulerian graphs are two classical classes of graphs
~\cite{We}. Chordal graphs constitute an important class of graphs in
algorithmic graph theory due to their connection to tree
decompositions~\cite{Go}. Chordal graphs are bridged and
1-hyperbolic. Chordal graphs are exactly the graphs admitting perfect
elimination orderings, i.e., linear orders $v_1,\ldots,v_n$ such that
for each $1<i\le n$ the neighbors of the vertex $v_i$ in $G_i$ induce
a clique. Perfect elimination orderings are dismantling and distance
preserving orders.  Dismantlable graphs are exactly the cop-win
graphs~\cite{NoWi}, i.e., the finite graphs in which the cop can
always catch the robber.  It is known that several classes of graphs
arising in Metric Graph Theory are dismantlable:
bridged~\cite{AnFa,Ch_bridged} and weakly bridged graphs~\cite{ChOs},
squares of graphs with convex balls~\cite{ChChGi}, and Helly
graphs~\cite{BaPe}. Dismantlability of graphs can be viewed as a
strong form of collapsibility of their clique complexes and can be
viewed as a tool to prove contractibility of such complexes. Deciding
if a graph $G$ is dismantlable is easy because it reduces to
recursively finding dismantlable vertices and removing them. Graphs
admitting distance preserving orderings have been introduced in
\cite{Ch_dpo}, where it was shown that several classes of graphs admit
DPO. In~\cite{CCHO} it is shown that all weakly modular graphs admit
DPO.  As shown in~\cite{CCMW}, the partial cubes admitting a distance
preserving ordering are exactly the ample partial cubes that admit a
corner peeling. Finally, it was proved in~\cite{CoDuNiSo} that the
problem of deciding if a graph has such a DPO is NP-complete.

\begin{proposition}\label{notFOLB} The following queries $\chordal, \planar, \dismantlable,\partialJ$, and $\eulerian$ are not FOLB-definable.
\end{proposition}

\begin{proof} The proofs are similar in spirit to the proof that the
  queries $\acyclicity$ and $\bipartite$ are not FOL-definable using
  Ehrenfeucht-Fra\"{i}ss\'e games, see~\cite{Kol,Lib}.

\medskip\noindent
{\bf To the query} {$\chordal$}.
For any $r \geq 1$, let $d\geq 3^{r+2}$. Consider the graphs $A=(V_1,E_1)$ and $B=(V_2, E_2)$,  where $A$ is a path $(x_1,x_2,\ldots,x_{4d-1},x_{4d})$ with $4d$ vertices
and $B$ is a union of a path $(y_1,y_2,\ldots,y_{2d-1},y_{2d})$ and a cycle $(y_{2d+1},\ldots,y_{4d},y_{2d+1})$ with $2d$ vertices each. Let $A^*$ be the graph obtained by adding a vertex $x^*$  and making it adjacent to all vertices of $A$ and let $B^*$ be the graph obtained
by adding a vertex $y^*$ and making it adjacent to all vertices in $B$; see Fig. \ref{A and B}. By construction,  $A^*$ is chordal and $B^*$ is not
chordal. Set  $V_1^*:= V_1 \cup \{x^*\}$ and $V_2^*:= V_2 \cup \{y^*\}$.

\begin{figure}[htb]
\centering
\hspace{10pt}
\includegraphics[width=70mm,scale=0.9]{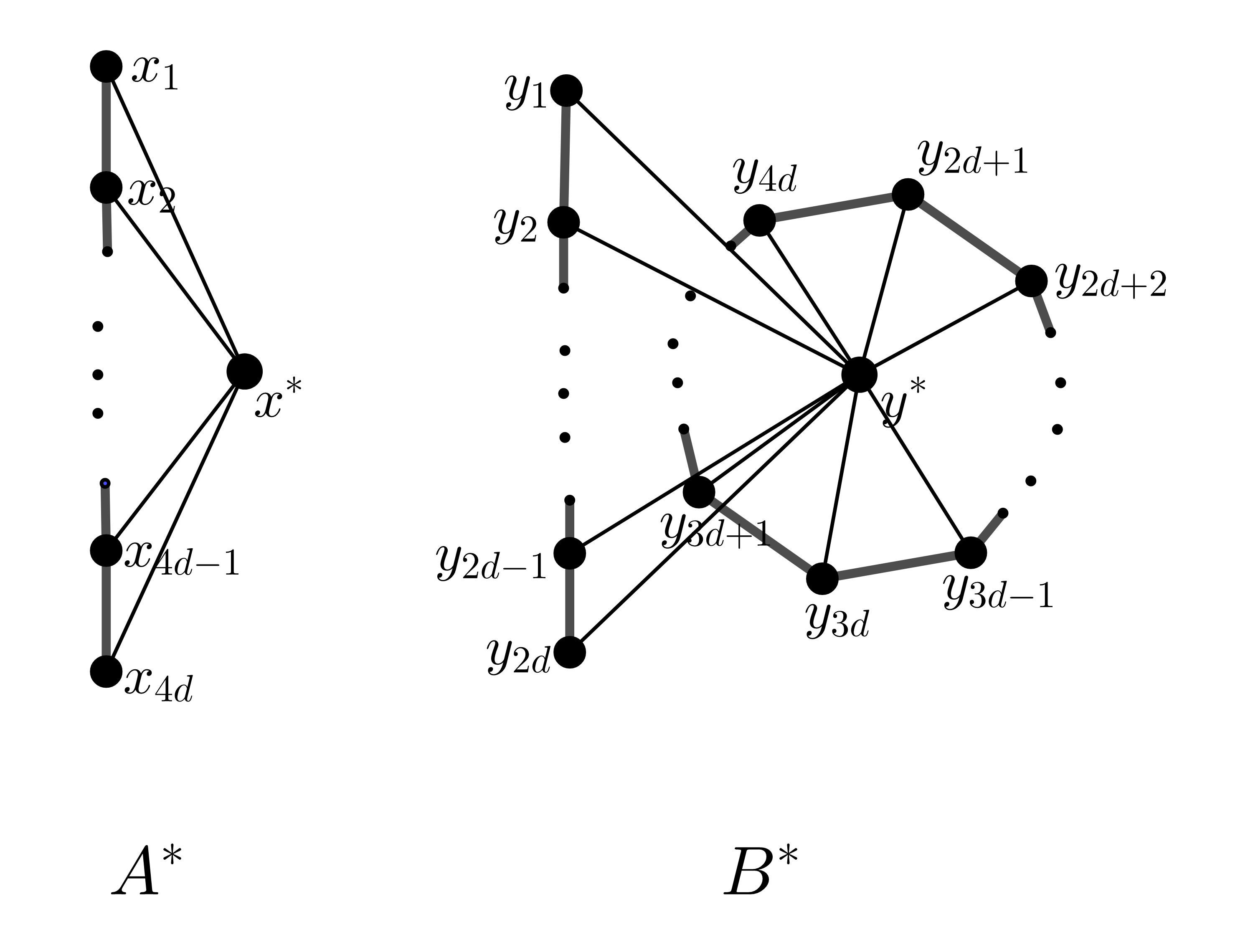}
\caption{The graphs $A^*$ and $B^*$ for the query {\sf Chordal}}
\label{A and B}
\end{figure}

Let $B_1$ and $B_2$ be the betweenness predicates of the graphs $A^*$
and $B^*$ and let $\bA^* = (V_1^*, B_1)$ and $\bB^* = (V_2^*, B_2)$
denote the graphic interval structures of $A^*$ and $B^*$.  The edges
of the graphs $A^*$ and $B^*$ are $E_{B_1}$ and $E_{B_2}$,
respectively. Notice that both graphs $A^*$ and $B^*$ have diameter 2
and that $x^*$ (respectively, $y^*$) is contained in all intervals
between two non-adjacent vertices $u,v$ of $A$ (respectively, of
$B$). Consequently, for any vertices $u,v \in A$ (respectively,
$u,v \in B$), we have $I(u,v) = \{u,v\}$ when $d_A(u,v) = 1$
(respectively, $d_B(u,v)=1$), $I(u,v) = \{u,x^*,v\}$ (respectively,
$I(u,v) = \{u,y^*,v\}$) if $d_A(u,v) \geq 3$ (respectively,
$d_B(u,v)\geq 3$), and $I(u,v) = \{u,x^*,z,v\}$ (respectively,
$I(u,v) = \{u,y^*,z,v\}$) if $d_A(u,v) = 2$ (respectively,
$d_B(u,v)=2$) and $z$ is the unique common neighbor of $u$ and $v$ in
$A$ (respectively, in $B$).
Using this structure of intervals in $A^*$ and
$B^*$, we can establish the following result:

\begin{claim}\label{partial_isomorphism}
  Let $f$ be a partial map  from $V_1^*$ to $V_2^*$ such that
  $f(x^*) = y^*$ if $x^*$ is in the domain $X$ of $f$ or if $y^*$ is
  in the codomain $Y$ of $f$.
If $f$ is an isomorphism between the subgraphs $H_1$ and
  $H_2$ induced by the sets $X$ and $Y$ in the graphs
  $A^*$ and $B^*$, respectively, then $f$ is a partial isomorphism between the
  interval structures $\bA^*$ and $\bB^*$.
\end{claim}

\begin{proof}
Since $f$ is an isomorphism from $H_1$ to $H_2$, observe that $f$ is
  necessarily injective.  To establish the claim, we have to show that
  for any $u,v,z$ of $X$, we have $B_1(u,z,v)$ if and only if
  $B_2(f(u),f(z),f(v))$. This trivially holds if $u = v$ or if
  $z \in \{u,v\}$. Since $f$ is injective, this also holds if
  $f(u) = f(v)$ or if $f(z) \in \{f(u),f(v)\}$.

  Thus suppose that $u \neq v$, $f(u) \neq f(v)$, $z \notin \{u,v\}$
  and $f(z) \notin \{f(u),f(v)\}$. Since $f$ is an isomorphism from
  $H_1$ to $H_2$, we have that $u \sim v$ if and only if
  $f(u) \sim f(v)$. In this case, since $z \notin \{u,v\}$ and
  $f(z) \notin \{f(u),f(v)\}$, neither $B_1(u,z,v)$ nor
  $B_2(f(u),f(z),f(v))$ holds. Assume now that $u \nsim v$ and
  $f(u) \nsim f(v)$. Since $u$ and
  $v$ (respectively, $f(u)$ and $f(v)$) are at distance $2$ in $A^*$
  (respectively, in $B^*$), we have $B_1(u,z,v)$ (respectively,
  $B_2(f(u),f(z),f(v))$) if and only if $z\sim u,v$ (respectively,
  $f(z) \sim f(u), f(v)$). Since $f$ is an isomorphism from $H_1$ to
  $H_2$, we have $z \sim u,v$ if and only if $f(z) \sim f(u),f(v)$ and
  consequently, $B_1(u,z,v)$ if and only if
  $B_2(f(u),f(z),f(v))$. This ends the proof of the claim.
\end{proof}

By Claim~\ref{partial_isomorphism}, to prove that the Duplicator wins
the $r$-move EF-games on the interval structures ${\bA}^*$ and
${\bB}^*$ it sufficed to prove that the Duplicator wins the $r$-move
EF-games on the graphs $A^*$ and $B^*$. We denote by $a_i$ and $b_i$
the vertices played in $A^*$ and $B^*$ at $i$th move of the
game. Since $x^*, y^*, x_1, x_{4d}, y_1, y_{2d}$ play a special role
in the graphs $A^*$ and $B^*$ (they are the only vertices that are not
of degree $4$), we set $a_0 = x^*$, $b_0 = y^*$, $a_{-1}=x_1$,
$b_{-1}= y_1$, $a_{-2}=x_{4d}$, and $b_{-2}=y_{2d}$.

We describe a strategy for Duplicator to ensure that at any step $i$
of the game, the following conditions hold for all $-2 \leq j \leq i$:
\begin{enumerate}[(1)]
\item $a_j = a_0 = x^*$ if and only if
  $b_j = b_0 = y^*$,
\item if $a_j \neq x^*$, then there exists an isomorphism $f$ from the
  ball $B_{3^{r-i}}(a_j,A)$ to the ball $B_{3^{r-i}}(b_j,B)$ such that
  $f(a_{\ell}) = b_\ell$ (respectively, $f^{-1}(b_{\ell})=a_{\ell}$)
  for any $a_{\ell} \in B_{3^{r-i}}(a_j,A)$ (respectively,
  $b_{\ell} \in B_{3^{r-i}}(b_j,B)$).
\end{enumerate}

Observe that by our choice of $d$, when $a_j \neq x^*$ and
$b_j \neq y^*$, the ball $B_{3^{r-i}}(a_j,A)$ (respectively,
$B_{3^{r-i}}(b_j,B)$) is a path that is an isometric subgraph of $A$
(respectively, of $B$) for any $0 \leq i \leq r$ and $-2 \leq j \leq i$.
Note also that condition (2) implies that:
\begin{enumerate}[(3)]
\item if $x^* \notin \{a_j,a_{\ell}\}$, then $d_A(a_j,a_{\ell}) > 3^{r-i}$
  if and only if $d_B(b_j,b_{\ell}) > 3^{r-i}$,
\item[(4)] if $x^* \notin \{a_j,a_{\ell}\}$ and
  $d_A(a_j,a_{\ell}) \leq 3^{r-i}$, then
  $d_A(a_j,a_{\ell}) = d_B(b_j,b_{\ell})$.
\end{enumerate}

Before the start of the game, we have $i=0$ and
the conditions (1) and (2) hold since
$d_A(a_{-2},a_{-1}) = 4d-1 \geq 4 \cdot 3^{r+2}-1 > 3^r$ and
$d_B(b_{-2},b_{-1}) = 2d-1 \geq 2 \cdot 3^{r+2}-1 > 3^r$.
Assume now that conditions (1) and (2) hold for the first $i$  steps
of the game and assume that
the Spoiler picks $a_{i+1}\in A$ (the case where the Spoiler picks
$b_{i+1} \in B$ is similar). First, suppose that
$a_{i+1}=a_j$ for some $j \leq i$ (this holds in particular if
$a_{i+1} \in \{x^*,x_1,x_{4d}\}$). In this case, set $b_{i+1} = b_j$
and observe that conditions (1) and (2) still hold after step $i+1$.
Thus, further suppose that $a_{i+1}\ne a_j$ for any $j \leq i$. We distinguish
two cases.

\begin{case} $d_A(a_{i+1},a_j) > 3^{r-(i+1)}$ for all
$-2 \leq j \leq i$ such that $a_j \neq x^*$.
\end{case}

For any $-2 \leq j \leq i$ such that $b_{j} \in \{y_1, \ldots, y_d\}$,
the ball $B_{3^{r-(i+1)}}(b_j,3^{r-(i+1)})$ contains at most
$2 \cdot 3^{r-(i+1)} + 1$ vertices. Consequently,
$\bigcup_{b_j \in \{y_1, \ldots, y_d\}}B_{3^{r-(i+1)}}(b_j,B)$
contains at most $(i+2)(2 \cdot 3^{r-(i+1)} + 1)$ vertices. Since
$2d \geq 2\cdot 3^{r+2} > (i+2)(2\cdot 3^{r-(i+1)} + 1)$, there exists
a vertex $y_k$ with $1 \leq k \leq 2d$ such that $y_k$ does not belong
to any of these balls of radius $3^{r-(i+1)}$, i.e., for any $b_j$
such that $b_j \in \{y_1, \ldots, y_d\}$, we have
$d_B(y_k,b_j) > 3^{r-(i+1)}$. Let Duplicator pick $b_{i+1}= y_k$ and
observe that condition (1) is satisfied with this choice of
$b_{i+1}$. We now show that condition (2) is also satisfied.  For any
$2 \leq j \leq i$ such that $a_j\neq x^*$, we have
$d_A(a_{i+1},a_j) > 3^{r-(i+1)}$ and $d_B(b_{i+1},b_j) >
3^{r-(i+1)}$. In particular, $d_A(a_{i+1},a_{-2})$,
$d_A(a_{i+1},a_{-1})$, $d_B(b_{i+1},b_{-1})$, $d_B(b_{i+1},b_{-1})$
are greater than $3^{r-(i+1)}$ and thus the balls
$B_{3^{r-(i+1)}}(a_{i+1},A)$ and $B_{3^{r-(i+1)}}(b_{i+1},B)$ are both
paths of length $3^{r-i}$ that do not contain any vertex $a_\ell$ or
$b_\ell$ with $\ell \leq i$, and thus condition (2) is satisfied for
$a_{i+1}$. For all other vertices, condition (2) trivially holds by
induction hypothesis.

\begin{case} There exists $-2 \leq j \leq i$ such that
$a_j\neq x^*$ and $d_A(a_j,a_{i+1}) \leq 3^{r-(i+1)}$.
\end{case}

By induction hypothesis, there is an isomorphism $f$ from the ball
$B_{3^{r-i}}(a_j,A)$ to the ball $B_{3^{r-i}}(b_j,B)$ such that for
any $a_\ell \in B_{3^{r-i}}(a_j,A)$ (respectively,
$b_\ell \in B_{3^{r-i}}(b_j,B)$), we have $f(a_{\ell}) = b_{\ell}$
(respectively, $f^{-1}(b_\ell) = a_\ell$). Let Duplicator pick
$b_{i+1} = f(a_{i+1})$ and note that condition (1) holds with such a
choice. We now show that condition (2) also holds. Observe that for
any $-2 \leq \ell \leq i+1$ such that $a_{\ell} \in A$ and
$a_{i+1} \in B_{3^{r-(i+1)}}(a_{\ell},A)$, the ball
$B_{3^{r-(i+1)}}(a_{\ell},A)$ is included in
$B_{3^{r-i}}(a_{j},A)$. Indeed, if
$d_A(a_{\ell},a_{i+1}) \leq 3^{r-(i+1)}$, we have
$d_A(a_{\ell},a_j) \leq d_A(a_\ell,a_{i+1}) + d_A(a_{i+1},a_{j}) \leq
2\cdot 3^{r-(i+1)}$ and thus $B_{3^{r-(i+1)}}(a_{\ell},A)$ is included
in $B_{3^{r-i}}(a_{\ell},A)$.  Similarly, since
$d_B(b_{i+1},b_j) = d_A(a_{i+1},a_j) \leq 3^{r-(i+1)}$, for any
$-2 \leq \ell \leq i+1$ such that $b_\ell \in B$ and
$b_{i+1} \in B_{3^{r-(i+1)}}(b_{\ell},B)$, the ball
$B_{3^{r-(i+1)}}(b_{\ell},B)$ is included in $B_{3^{r-i}}(b_{j},B)$.
Therefore, $a_{i+1} \in B_{3^{r-(i+1)}}(a_{\ell},A)$ if and only if
$b_{i+1} \in B_{3^{r-(i+1)}}(b_{\ell},B)$ and when it is the case, $f$
induces an isomorphism between the balls $B_{3^{r-(i+1)}}(a_{\ell},A)$ and
$B_{3^{r-(i+1)}}(f(a_{\ell}),B) = B_{3^{r-(i+1)}}(b_{\ell},B)$. By
induction hypothesis and since $f(a_{i+1}) = b_{i+1}$, for any
$-2 \leq k \leq i+1$ such that $a_{k} \in B_{3^{r-(i+1)}}(a_{\ell},A)$
(respectively, $b_{k} \in B_{3^{r-(i+1)}}(b_{\ell},B)$), we have
$f(a_k) = b_k$ (respectively, $f^{-1}(b_k) = a_k$). This establishes
condition (2) for all balls containing $a_{i+1}$. For the other balls,
the result holds trivially by induction hypothesis.

\medskip
At the end of the game, i.e., after $r$ rounds, for any
$-2 \leq j,\ell \leq r$, we show that $a_j\sim a_{\ell}$ in $A^*$ if
and only if $b_j \sim b_{\ell}$ in $B^*$. Note that by conditions (1)
and (2), we have $a_j = a_\ell$ if and only if $b_j = b_\ell$. In this
case, we have $a_j \nsim a_{\ell}$ and $b_j \nsim b_{\ell}$. Assume
now that $a_j \neq a_{\ell}$ and $b_j \neq b_{\ell}$.  If $a_j = x^*$
(respectively, $a_{\ell} = x^*$, $b_j = y^*$, $b_{\ell}=y^*$), then by
  condition (1), we have $b_j = y^*$ (respectively, $b_{\ell} = y^*$,
    $a_j = x^*$, $a_{\ell}=x^*$) and in this case we have
    $a_j \sim a_{\ell}$ and $b_j \sim b_{\ell}$ since $x^*$ and $y^*$
    are respectively adjacent to all vertices of $A$ and $B$. Suppose
    now that $a_j, a_\ell \in A$, $b_j, b_{\ell} \in B$, and
    $a_j \sim a_\ell$ (respectively, $b_j \sim b_\ell$). By Condition
    (2), there is an isomorphism $f$ from $B_1(a_j,A)$ to
    $B_1(b_j, B)$ such that $f(a_\ell) = b_\ell$ (respectively,
    $f^{-1}(b_\ell) = a_{\ell})$. Consequently,
    $b_\ell=f(a_\ell) \sim b_j=f(a_j)$ (respectively,
    $a_\ell=f^{-1}(b_\ell) \sim a_j=f^{-1}(b_j)$ ) since
    $a_\ell \sim a_j$ (respectively, $b_{\ell} \sim b_j$).

    Consequently, the map $f : a_j \mapsto b_j$ defines an isomorphism
    between the graphs induced by $X = \{a_{-2},a_{-1}, \ldots,a_r\}$
    in $A^*$ and $Y = \{b_{-2},b_{-1}, \ldots,b_r\}$ in $B^*$. By
    Claim~\ref{partial_isomorphism}, we get that the Duplicator wins
    the $r$-moves EF-game on the interval structure
    ${\bA}^*=(V_1^*,B_1)$ and ${\bB}^*=(V_2^*,B_2)$. By Theorem
    \ref{EhrFra_bis}, {$\chordal$} is not FOLB-definable.

\medskip\noindent
{\bf To the query} {$\dismantlable$.} We consider the same graphs $A=(V_1,E_1)$ and $B=(V_2,E_2)$ as in query { \sf $\chordal$}.
Let $A^*$ be the graph obtained by adding two non-adjacent  vertices $x^*_1$ and $x^*_2$ and making them adjacent to all vertices in $A$
and let $B^*$ be the graph obtained by adding two non-adjacent vertices $y^*_1$ and $y^*_2$ and making them adjacent to all vertices of $B$;
see Fig. \ref{dismantlable}.  Then $x_1,x_2,\ldots,x_{4d-1},x^*_1,x^*_2,x_{4d}$ is a dismantling order of $A^*$ because each vertex $x_i, i<4d$
is dominated in the subgraph induced by the remaining vertices by  the vertex $x_{i+1}$ while the subgraph induced by the last three
vertices $x^*_1,x^*_2,x_{4d}$ is a 2-path with $x_{4d}$ as the middle vertex. On the other hand,
the graph  $B^*$ is not dismantlable because the unique way to partially dismantle $B^*$ is to remove the vertices of the path of $B_r$ either in order $y_1,y_2,\ldots,y_{2d}$
or in the reverse order  $y_{2d},y_{2d-1}\ldots,y_{1}$. The resulting subgraph of $B^*$ is a double wheel induced by the cycle $(y_{2d+1},\ldots,y_{4d})$ and two non-adjacent vertices $y_1^*$ and $y^*_2$
and this subgraph does not contain any dominated vertex.

\begin{figure}[htb]
\centering
\includegraphics[width=90mm,scale=1.]{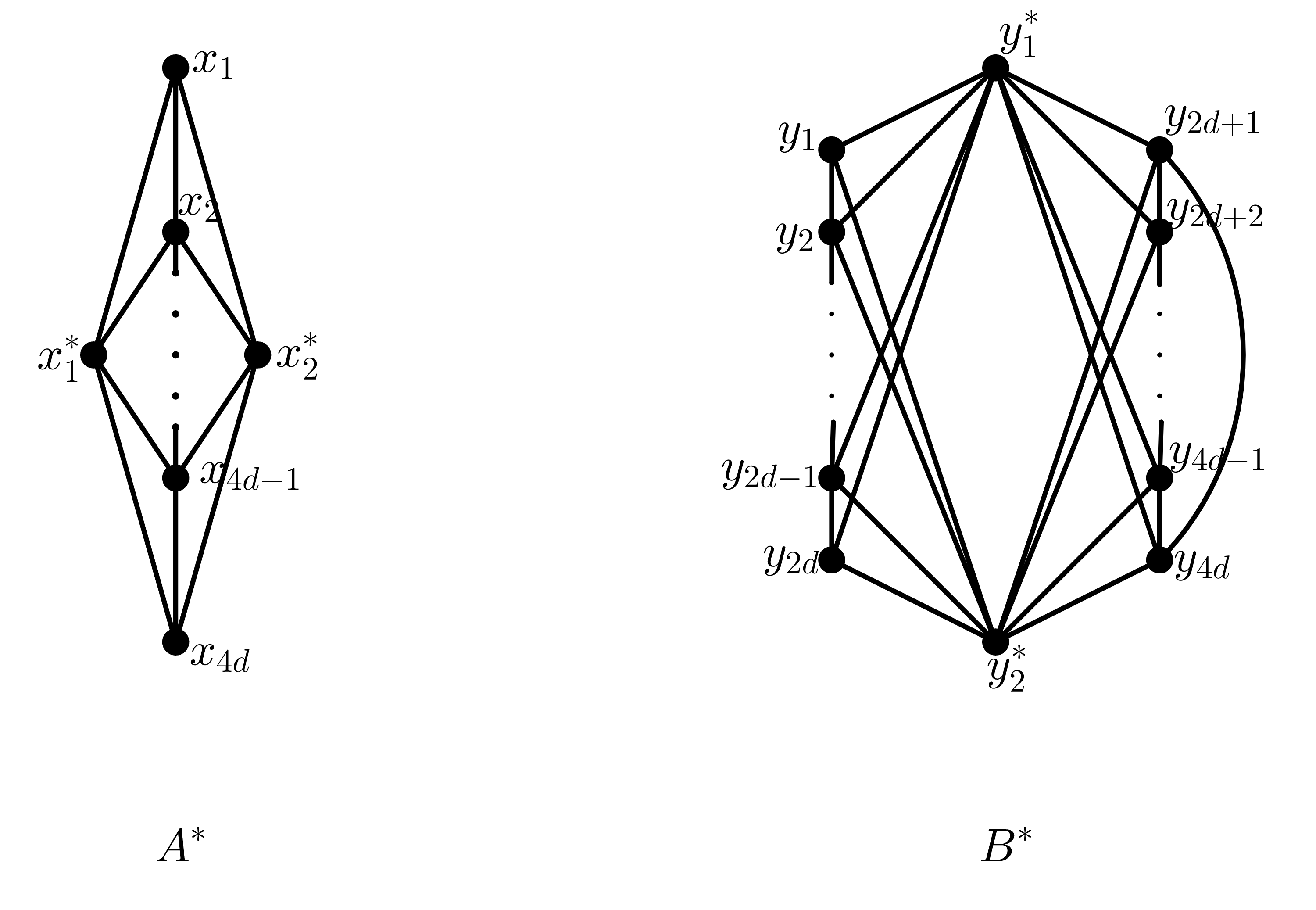}
\hspace{10pt}
\caption{The graphs $A^*$ and $B^*$ for the query {\sf Dismantlable}}
\label{dismantlable}
\end{figure}

Let $B_1$ and $B_2$ be the betweenness predicates of the graphs $A^*$
and $B^*$. Let ${\bA}^* = (V_1^* ,B_1)$ and $\bB^*=(V_2^* ,B_2)$,
where $V_1^* = V_1 \cup \{x^*_1, x^*_2 \}$ and
$V_2^* = V_2 \cup \{y^*_1, y^*_2 \}$.  Note that both graphs $A^*$ and
$B^*$ have diameter 2 and that $x_1^*, x_2^*$ (respectively,
$y_1^*, y_2^*$) are contained in all intervals between non-adjacent
vertices $u,v$ of $A$ (respectively, of $B$).  Consequently, for any
vertices $u,v \in A$ (respectively, $u,v \in B$), we have
$I(u,v) = \{u,v\}$ when $d_A(u,v) = 1$ (respectively, $d_B(u,v)=1$),
$I(u,v) = \{u,x_1^*,x_2^*,v\}$ (respectively,
$I(u,v) = \{u,y_1^*,y_2^*,v\}$) if $d_A(u,v) \geq 3$ (respectively,
$d_B(u,v)\geq 3$), and $I(u,v) = \{u,x_1^*,x_2^*,z,v\}$ (respectively,
$I(u,v) = \{u,y_1^*,y_2^*,z,v\}$) if $d_A(u,v) = 2$ (respectively,
$d_B(u,v)=2$) and $z$ is the unique common neighbor of $u$ and $v$ in
$A$ (respectively, in $B$). Observe also that
$I(x_1^*,x_2^*) = V(A^*)$ and $I(y_1^*,y_2^*) = V(B^*)$.

Using this structure of intervals and a proof similar to the proof of
Claim \ref{partial_isomorphism}, one can establish that for any map
$f$ from $X \subseteq V_1^*$ to $Y \subseteq V_2^*$ such that
$f(x_1^*) = y_1^*$ if $x_1^* \in X$ or $y_1^* \in Y$,
$f(x_2^*) = y_2^*$ if $x_2^* \in X$ or $y_2^* \in Y$, if $f$ is an
isomorphism between the subgraphs induced by $X$ and $Y$ in the graphs
$A^*$ and $B^*$, then $f$ is a partial isomorphism between the
interval structures $A^*$ and $B^*$. Then, using a proof similar as in
the case of chordal graphs, we can establish that Duplicator wins the
EF-games on $\bA^*$ and $\bB^*$. By Theorem~\ref{EhrFra_bis}, this
implies that $\dismantlable$ is not FOLB-definable.

\medskip\noindent {\bf To the query} {\sf $\planar$.} We consider the
same graphs $A=(V_1,E_1)$ and $B=(V_2,E_2)$ as in queries $\chordal$
and $\dismantlable$ and the same graphs $A^*$ and $B^*$ in the query
$\dismantlable$. The graph $A^*$ is planar: a planar drawing of $A^*$
is given in Fig. \ref{dismantlable}. On the other hand, $B^*$ is not
planar because it contains $K_5$ as a minor. This $K_5$-minor is
defined by any three distinct vertices $y_i,y_j,y_k$ of the cycle $C$
of $B$ and the vertices $y^*_1,y^*_2$. The vertices $y_i,y_j,y_k$ are
connected by three disjoint subpaths of the cycle $C$, $y^*_1$ and
$y^*_2$ are connected to $y_i,y_j,y_k$ by edges, and $y_1^*$ and
$y^*_2$ are connected by a path of length 2 passing via the path of
$B$. The result then follows from the fact that the Duplicator wins
the EF-game on $\bA^*$ and $\bB^*$.

\medskip\noindent {\bf To the query} {\sf
  $\partialJ$.} Consider the graphs $A= (V_1,E_1)$ and $B=(V_2, E_2)$, where $A$ is a
cycle $(x_1,x_2,\ldots,x_{2d-1},x_{2d}, x_1)$ with $2d$ vertices and
$B$ consists of two odd cycles $(y_1,y_2,\ldots,y_{d-1},y_{d})$ and
$(y_{d+1},\ldots,y_{2d},y_{d+1})$. Let $A^*$ be the graph obtained by
adding a vertex $x^*$ and making it adjacent to all vertices of $A$
and let $B^*$ be the graph obtained by adding a vertex $y^*$ and
making it adjacent to all vertices in $B$. One can directly check (or
use the result of~\cite{Ch_Johnson}) that $A^*$ is a partial Johnson
graph and $B^*$ is not a partial Johnson graph because $B^*$ contains
odd wheels. Since the structure of intervals in this pairs of graphs
is similar to the graphs considered in the case of the query {\sf
  $\chordal$}, an analogous of Claim \ref{partial_isomorphism} for
this pair of graphs also holds. Then, using the same proof as in the
case of chordal graphs (where the vertices
$a_{-2},a_{-1},b_{-2}, b_{-1}$ are not defined), we can establish that
Duplicator wins the EF-games on $\bA^*$ and $\bB^*$. By
Theorem~\ref{EhrFra_bis}, this implies that $\dismantlable$ is not
FOLB-definable.
\begin{center}
\begin{figure}
\includegraphics[width=60mm,scale=0.85]{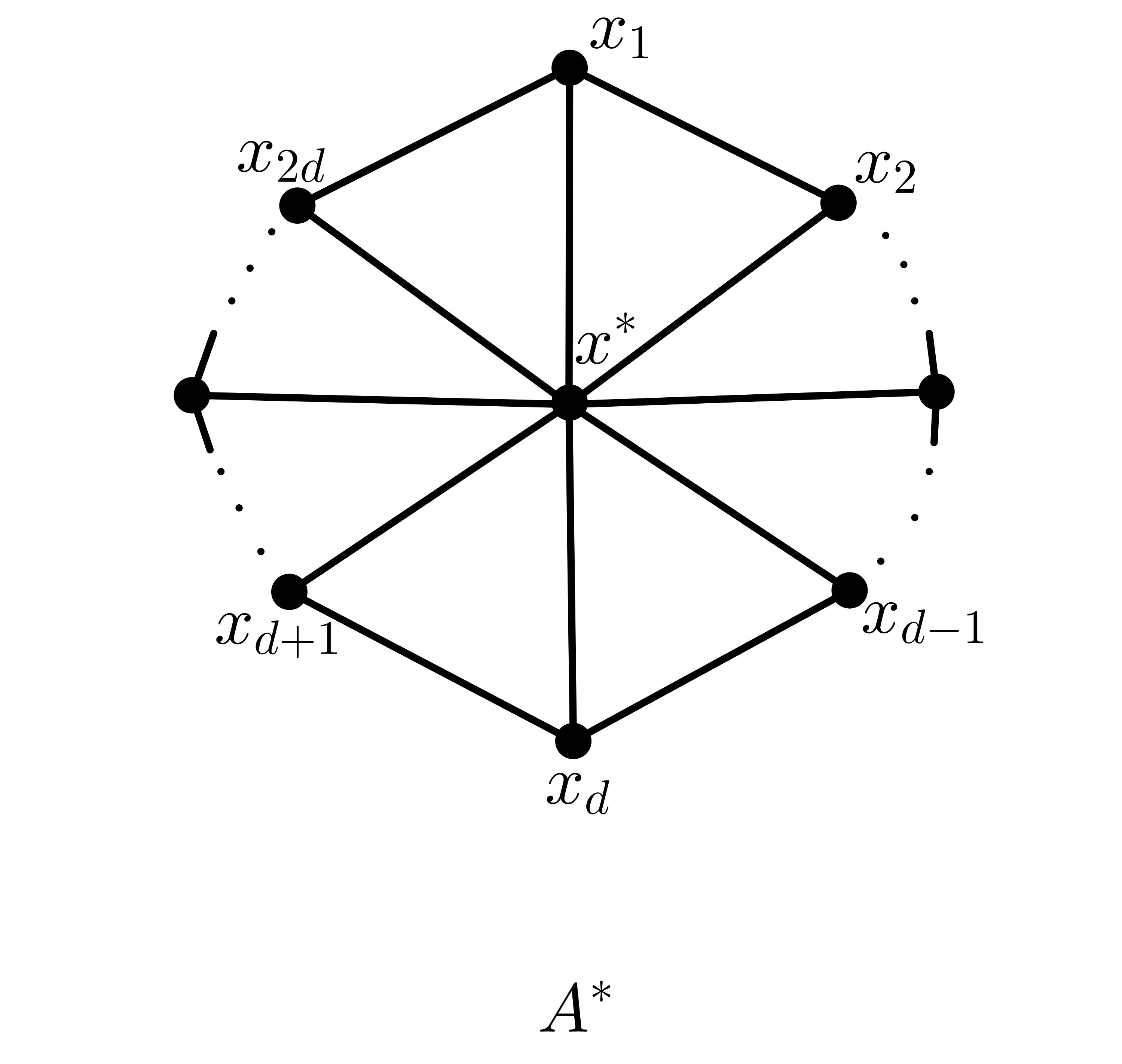}
\hspace{10pt}
\includegraphics[width=50mm,scale=0.85]{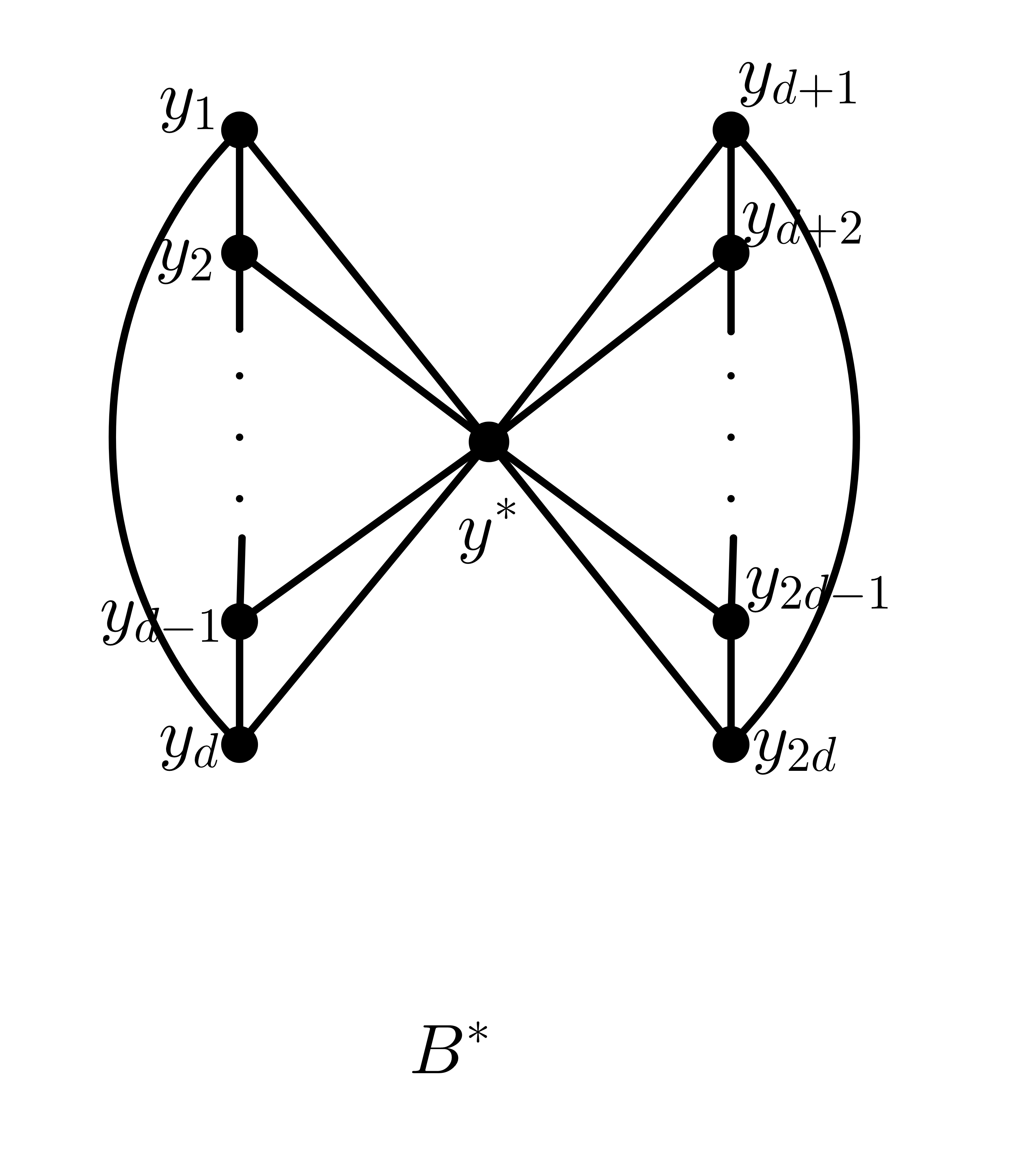}
\caption{The graphs $A^*$ and $B^*$ for the query {\sf Partial Johnson}}
\label{partial_johnson}
\end{figure}
\end{center}

\medskip\noindent {\bf To the query} {\sf $\eulerian$.} For any
$r \geq 1$, consider the graphs $A=(V_1,E_1)$ and $B=(V_2,E_2)$, where
$A=\overline{K}_{2r}$ is a stable set of size $2r$ and
$B=\overline{K}_{2r+1}$ is a stable set of size $2r+1$.  Let $A^*$ be
the graph obtained by adding two non-adjacent vertices $x^*_1$ and
$x^*_2$ and making them adjacent to all vertices in $A_r$ and let
$B^*$ be the graph obtained by adding two non-adjacent vertices
$y^*_1$ and $y^*_2$ and making them adjacent to all vertices of $B_r$.
Clearly, $A^*$ and $B^*$ are the complete bipartite graphs $K_{2,2r}$
and $K_{2,2r+1}$. Let
$V_1=\{ x_1,\ldots,x_{2r}\}, V_2=\{ y_1,\ldots,y_{2r},y_{2r+1}\}$ and
$V_1^* = V_1 \cup \{x^*_1, x^*_2 \}, V_2^*= V_2 \cup \{y^*_1, y^*_2
\}$.  Obviously, $A^*$ is an Eulerian graph and $B^*$ is not Eulerian.
Let $B_1$ and $B_2$ be the betweenness predicates of the graphs $A^*$
and $B^*$ and consider the interval structures $\bA^*=(V_1^*,B_1)$ and
$\bB^*=(V_2^* ,B_2)$.  Both graphs $A^*$ and $B^*$ have diameter 2 and
one easily see that for any two vertices $u,v$ of $A$ (respectively,
$B$), we have $I(u,v) = \{ u, x^*_1, x^*_2, v \}$ (respectively,
$I(u,v)=\{ u, y^*_1, y^*_2, v \}$). Observe also that
$I(x_1^*,x_2^*) = V_1^*$ and $I(y_1^*,y_2^*) = V_2^*$.

Using this structure of intervals and a proof similar (but simpler) to
the proof of Claim \ref{partial_isomorphism}, one can establish that
for any map $f$ from $X \subseteq V_1^*$ to $Y \subseteq V_2^*$ such
that $f(x_1^*) = y_1^*$ if $x_1^* \in X$ or $y_1^* \in Y$,
$f(x_2^*) = y_2^*$ if $x_2^* \in X$ or $y_2^* \in Y$, if $f$ is an
isomorphism between the subgraphs induced by $X$ and $Y$ in the graphs
$A^*$ and $B^*$, then $f$ is a partial isomorphism between the
interval structures $A^*$ and $B^*$. Then, using a proof similar as in
the case of chordal graphs, we can establish that Duplicator wins the
EF-games on $\bA^*$ and $\bB^*$. By Theorem~\ref{EhrFra_bis}, this
implies that $\eulerian$ is not FOLB-definable.
\end{proof}

Since FOLB-queries can be tested in polynomial time on any graph $G$
(see Section~\ref{FO-poly}), any property that is NP-hard to recognize
is unlikely to be expressible as a FOLB-query. It is known that it is
NP-complete to decide whether a given graph admits a distance
preserving order~\cite{CoDuNiSo}. Therefore if P$\neq$NP, $\dpo$
cannot be expressed as a FOLB-query. In fact, one can show that $\dpo$
is not expressible as a FOLB-query. To establish this result, one can
proceed as in the proof of Proposition~\ref{notFOLB} in order to prove
that Duplicator wins the games on the two graphs $A^*$ and $B^*$ from
Figure~\ref{dpo} ($A^*$ has a distance preserving order while $B^*$
does not). Since these graphs have diameter 4, in order to give a
formal proof, we need to adapt Claim~\ref{partial_isomorphism} and the
inductive condictions (1) and (2) used in the proof of
Proposition~\ref{notFOLB}. Notice that in these graphs, we have nine
different types of intervals. We do not give a proof of this result,
because this would need a lengthy proof leading to an expected
conclusion.

\begin{figure}[htb]
\centering
\includegraphics[width=70mm,scale=1.6]{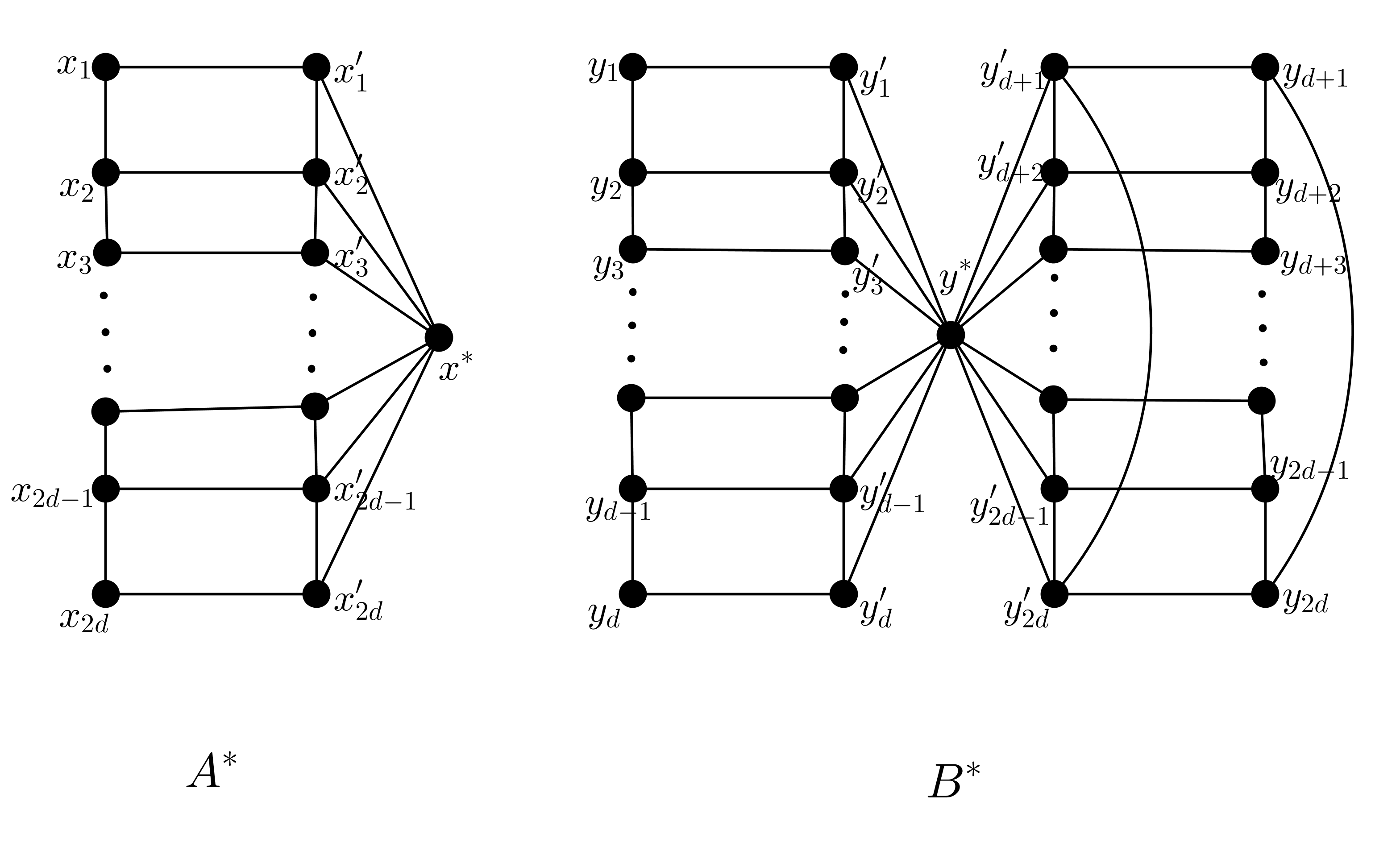}
\caption{The graphs $A^*$ and $B^*$ for the query $\dpo$}
\label{dpo}
\end{figure}

\section{Running time of FOLB queries}\label{FO-poly}

Let $\phi$ be an FO sentence in vocabulary $\sigma$, and let $\bA$ be a  $\sigma$-structure. Encoding of a formula $\phi$, $enc(\phi)$,
could be its syntactic tree represented as a string. The length of this string, $enc(\phi)$, is denoted by $||\phi||$.
We define the encoding $enc(\bA)$ of a structure  $\bA$ as the concatenation of $0^n 1$
and all the $enc(R_i)$, where $n$ is the cardinality of the domain of $\bA$ and $R_i$, for each $i=1,2,\ldots ,l$
is the relation symbol in the signature of $\bA$. That is, $enc(\bA) = 0^n1. enc(R_1)\cdots enc(R_l)$. The length of
this string is $||\bA|| = (n+1)+ \sum_{i=1}^p n^{arity(R_i)}$. The \emph{width} of an FO formula $\phi$
is the maximum number of free variables in a subformula of $\phi$. The running time of FO queries in terms of the size of encodings
of a FO query and a structure is given in \cite{Lib} as follows:

\begin{proposition}[\!\!\cite{Lib}] \label{Libkin}
Let $\phi$ be an FO sentence in vocabulary $\sigma$ and let $\bA$ be a  $\sigma$-structure.
If the width of $\phi$ is $k$, then checking whether $\bA$ is a model of $\phi$ can be done in time $O(||\phi || \times || \bA||^k)$.
\end{proposition}

This gives a polynomial time algorithm for evaluating FO queries on finite structures, for a fixed sentence $\phi$.

The interval structure $\bG=(V,B)$ with $|V|=n$ and $B$, being a ternary relation, can be encoded by a string of length
$|| \bG || = (n+1+n^3)$. Furthermore, for the fixed formula $\phi $, the $||\phi ||$ is a constant, say $c$. The total number of variables
used in $\phi$ will be the upper bound for $k$. Note that  every FOLB sentence defined in this paper uses a fixed number of
variables and, in some cases, the subformulas given by the query $\subgraph$ or  $\isometric$ in sentence $\phi$ maximum contributes
to the width of $\phi$. Thus for a fixed  sentence $\phi$ in FOLB, checking whether $\bA$ is a model of $\phi$ can be done in
time  $O(||\phi || \times || \bA||^k)$ $=$ $O(c\times (n+1+n^3)^k)=O(c\times n^{3k})$.
Note that a particular FOLB-query for a graph class  $\mathcal{C}$ actually characterizes  $\mathcal{C}$.
The \emph{recognition} of a graph class $\mathcal{C}$ is defined as the decision problem ``does a given graph $G$ belong the graph class  $\mathcal{C}$?".
The input is an arbitrary  graph $G$ and the output is {\sf True} if and only if $G$ belongs to $\mathcal{C}$ and {\sf False} otherwise.
If there is a corresponding FOLB-query characterising the graph class $ \mathcal{C}$, then for a given graph $G$, we can ensure that there exists
an algorithm to check whether $G$ is in the graph class $\mathcal{C}$ using that FOLB query.
From all these observations we obtain the following remark: 

\begin{proposition} \label{recognition}
All FOLB-definable graph classes given in Theorem \ref{theorem1}
can be recognized in polynomial time.
\end{proposition}

Notice that most of the classes of graphs from  Theorem \ref{theorem1} can be recognized by more efficient algorithms than those using Propositions \ref{Libkin} and \ref{recognition}.

\section{Conclusion} In this paper, we showed that most of graph classes investigated in Metric Graph Theory are definable in FOLB. On the other hand,
we showed that several sub- or super-classes of  them  are not definable in FOLB.  Tarski's theory uses the predicates ``betweenness'' and ``congruence''.
For graphs, the {\it congruence} is the quaternary relation $\equiv$ meaning that $uv\equiv u'v'$ if $d(u,v)=d(u',v')$ and we can call the resulting logic
the  \emph{Fist Order Logic with Betweenness and Congruency}, abbreviated  \emph{FOLBC}. However, for graphs it seems that FOLBC is not  more expressible than
FOLB. At least the graph classes not expressible in FOLB remains not expressible in FOLBC and for all classes for which FOLB-definability is open the
FOLBC-definability is open as well. On the other hand, if instead of the congruency axiom one can compute distances between vertices of a graph, then
the property  ``the hyperbolicity of a graph is at most $\delta$'' becomes easily expressible. Distance computation and angle measurement (Scale and Protractor)
is at the basis of classical metric approach to plane geometry due to Birkhoff \cite{Birkhoff} (see the book \cite{MiPa}).  Note that FOLB can be extended to
disconnected graphs \cite{cha18}. Another perspective is to consider the Monadic Second Order Logic with Betweenness for graphs. Finally, an important long
term research perspective is the algorithmic status of FOLB model checking
on  FOLB-definable classes of graphs: : for a FOLB-definable
class $\C$, can we find an algorithm that given a FOLB formula $\phi$
and a graph $G \in \C$ on $n$ vertices check whether $G$ satisfies
$\phi$ in time $f(\phi)\cdot \poly(n)$ where the
polynomial in $n$ is independent of $\phi$.

\end{document}